\documentclass{amsart}
\input{epsf.tex}
\usepackage[dvips]{graphicx,color}
\usepackage{amssymb,euscript,epsfig}

\newcommand{\calG}{\mathcal{G}}

\newcommand{\calP}{\mathcal{P}}
\newcommand{\calQ}{\mathcal{Q}}

\newcommand{\calT}{\mathcal{T}}
\newcommand{\T}{\calT}

\newcommand{\euC}{\EuScript{C}}
\newcommand{\euS}{\EuScript{S}}

\newcommand{\MF}{\ensuremath{\mathcal{MF}}}
\newcommand{\ML}{\ensuremath{\mathcal{ML}}}
\newcommand{\PMF}{\ensuremath{\mathcal{PMF}}}
\newcommand{\PML}{\ensuremath{\mathcal{PML}}}

\renewcommand{\H}{\ensuremath{\mathbb{H}}}
\newcommand{\C}{\ensuremath{\mathbb{C}}}
\newcommand{\R}{\ensuremath{\mathbb{R}}}
\newcommand{\CC}{\ensuremath{{\sf C}}}
\newcommand{\PCC}{\ensuremath{\mathcal P{\sf C}}}
\newcommand{\RP}{{\R\rm{P}}}

\newcommand{\G}{\ensuremath{G}}

\DeclareMathOperator{\core}{core}
\DeclareMathOperator{\hull}{hull}
\DeclareMathOperator{\I}{i}
\DeclareMathOperator{\CV}{\sf CV}
\DeclareMathOperator{\Mod}{Mod}

\DeclareMathOperator{\CAT}{CAT}
\DeclareMathOperator{\SL}{SL}

\DeclareMathOperator{\Real}{Re}
\DeclareMathOperator{\Ext}{Ext}

\newcommand{\Teich}{Teich\-m\"{u}ller~}

\newcommand{\QQ}{{\calQ^1}}

\newcommand{\cyl}{\ensuremath{{\sf cyl}}}
\newcommand{\Flat}{\ensuremath{{\sf Flat}}}
\newcommand{\Mix}{\ensuremath{{\sf Mix}}}
\newcommand{\PMix}{\ensuremath{\mathcal P{\sf Mix}}}
\newcommand{\foltheta}{\ensuremath{{\nu_q^\theta}}}

\newcommand{\from}{{\colon \,}}
\newcommand{\limni}{\lim\limits_{n\to\infty}}

\renewcommand{\matrix}[4]{\bigl(\begin{smallmatrix} #1& #2
\\#3&#4\end{smallmatrix}\bigr)}

\theoremstyle{plain}
\newtheorem{theorem}{Theorem}
\newtheorem{corollary}[theorem]{Corollary}
\newtheorem{proposition}[theorem]{Proposition}
\newtheorem{lemma}[theorem]{Lemma}
\newtheorem*{question}{Question}
\newtheorem*{rigid-iff-dense}{Theorem \ref{T:rigid iff dense}}
\newtheorem*{current-embedding}{Theorem \ref{T:current embedding}}
\newtheorem*{boundary}{Theorem \ref{T:boundary}}
\newtheorem*{S-rigid}{Theorem \ref{T:S rigid}}
\newtheorem*{nonrigid}{Theorem \ref{T:nonrigid}}

\theoremstyle{remark}

\theoremstyle{definition}
\newtheorem{remark}[theorem]{Remark}

\begin{document}

\title{Length spectra and degeneration of flat metrics}

\author[Duchin]{Moon Duchin}
\address{Department of Mathematics\\  
530 Church Street\\
Ann Arbor, MI 48109-1043
}
\email{mduchin@umich.edu}
\urladdr{http://www.math.lsa.umich.edu/$\sim$mduchin/}

\author[Leininger]{Christopher J. Leininger}
\address{Department of Mathematics \\
University of Illinois at Urbana-Champaign \\
1409 W. Green St. \\
Urbana, IL 61801
}
\email{clein@math.uiuc.edu}
\urladdr{http://www.math.uiuc.edu/$\sim$clein/}

\author[Rafi]{Kasra Rafi}
\address{Department of Mathematics \\
University of Oklahoma\\
Norman, OK 73019-0315
}
\email{rafi@math.ou.edu}
\urladdr{http://www.math.ou.edu/$\sim$rafi/}

\date{\today}
\maketitle

\begin{abstract}
In this paper we consider flat metrics (semi-translation structures) on surfaces of finite type.  There are two main results.
The first is a complete description of when a set of simple closed curves is spectrally rigid, 
that is, when the length vector determines a metric among the class of flat metrics.
Secondly, we give an embedding into the space of geodesic currents and use this 
to get a boundary for the space of flat metrics.  The geometric interpretation is that flat metrics
degenerate to {\em mixed structures} on the surface:  part flat metric and part measured
foliation.
\end{abstract}

%%%%%%%%%%%%%%%%%%%%%%
\section{Introduction}
%%%%%%%%%%%%%%%%%%%%%%

From the lengths of all, or some, curves on a surface $S$, can you
identify the metric?  To be precise, fix a finite-type surface $S$, denote by $\euC(S)$ the set of homotopy classes
of closed curves on $S$, and
let $\euS(S)$ be the homotopy classes represented by simple closed curves (simply 
denoted by $\euC$ and $\euS$ when $S$ is understood).
Given an isotopy class of metrics $\rho$ and a curve $\alpha \in \euC$, we
write $\ell_\rho(\alpha)$ to denote the infimum of lengths of
representatives of $\alpha$ in a representative metric for $\rho$, and
we call this the \textit{length of $\alpha$ in $\rho$} or the
\textit{$\rho$--length of
$\alpha$}.  For a set
of curves $\Sigma\subset\euC$, we define the \textit{(marked)
$\Sigma$--length spectrum of $\rho$} to be the length vector, indexed over $\Sigma$:
$$\lambda_\Sigma(\rho) = (\ell_\rho(\alpha))_{\alpha \in
\Sigma}
\in \mathbb R^\Sigma.$$

For a family of metrics $\calG=\calG(S)$, up to isotopy, and a family of
curves  $\Sigma$, we
are interested in the problem of deciding when
$\lambda_\Sigma(\rho)$ determines $\rho$.  In other words, we ask

\begin{question} Is the map $\calG \to \mathbb R^\Sigma$ given by
$\rho
\mapsto \lambda_\Sigma(\rho)$ an injection?
\end{question}
\noindent If this map is injective, so that $\rho\in\calG$ is determined by the lengths of
the $\Sigma$ curves, we say that $\Sigma$ is  {\em spectrally rigid} over $\calG$.

For instance, we may take $\Sigma = \euS$, and $\calG= \T(S)$,
the \Teich space of complete finite-area hyperbolic metrics on $S$ (constant curvature $-1$).
Here it is a classical fact due to Fricke that $\T(S)
\to \mathbb R^{\euS}$ is injective; that is, $\euS$ is spectrally rigid over $\T(S)$.

Another natural family of metrics arising in \Teich theory consists of
those induced by unit-norm quadratic differentials;
these are locally flat (isometrically Euclidean)
away from a finite number of singular points with cone angles $k\pi$.  
We note that these are nonpositively
curved in the sense of comparison geometry when the surface $S$ is closed, but $k=1$ is allowed 
in the case of punctures.
We will call these {\em flat metrics} on $S$ (see Section \ref{preliminary section} for a detailed discussion).
For example, identifying opposite sides of a regular Euclidean octagon produces a
flat metric on a genus-two surface, with the negative curvature concentrated into one cone point 
of angle $6\pi$.
We denote this family of metrics by $\Flat(S)$.

\begin{theorem}  \label{T:S rigid}
For any finite type surface $S$, the set of simple closed curves $\euS$ is spectrally rigid over
$\Flat(S)$.
\end{theorem}

Put in other terms, this theorem states that 
the lengths of simple closed curves determine a quadratic differential up to rotation.  
Let $\xi(S) = 3g-3 + n$ be a measure of the complexity of $S$,
where $g$ is the genus and $n$ is the number of punctures.
Then we can compare Theorem~\ref{T:S rigid} to the rigidity over hyperbolic metrics
by noting that the dimension of $\T(S)$ is $2\xi$, while the dimension of $\Flat(S)$ is $4\xi-2$.  

In fact, we obtain a much sharper version of Theorem \ref{T:S rigid} which provides a complete answer to the motivating question above for simple closed curves over flat metrics.  Let 
$\PMF=\PMF(S)$ denote Thurston's space of projective measured foliations on $S$.

%%%%%%%ORIGINAL
\begin{theorem}\label{T:rigid iff dense}
If $\xi(S) \geq 2$, then $\Sigma\subset\euS\subset\PMF$ is spectrally rigid over $\Flat(S)$ if
and only if $\Sigma$ is dense in $\PMF$.
\end{theorem}

This theorem stands in contrast to the hyperbolic case, 
where there are {\em finite} spectrally rigid sets, as is further discussed in  \S \ref{S:context}.  We also remark 
that if $\xi(S) \leq 1$ then it is easy to see that any set of three distinct, primitive curves is spectrally rigid 
over $\Flat(S)$; see Proposition \ref{P:three curves torus}.

%%%%%%%ORIGINAL
\begin{theorem} \label{T:nonrigid}
Suppose $\xi(S) \geq 2$.  
If $\Sigma \subset \euS \subset \PMF$ and $\overline \Sigma \neq \PMF$, then there is a deformation family $\Omega_\Sigma \subset \Flat(S)$ for which 
$\Omega_\Sigma \to \mathbb R^\Sigma$ is constant, and such that
the dimension of $\Omega_\Sigma$ is proportional to the dimension of $\Flat(S)$ itself.
\end{theorem}

In particular, in the closed case, our construction produces $2g-3$ parameters for deformations,
while the dimension of $\Flat(S)$ in this case is $12g-14$.  

Another result needed for the proof of Theorem \ref{T:rigid iff dense} is a version of Thurston's theorem that the hyperbolic lengths for simple closed curves continuously extends to the space 
$\MF(S)$ of measured foliations (or laminations) on $S$.  In \cite{bon-currents}, Bonahon gave a very elegant proof of this (for closed surfaces) based on a unified approach to studying hyperbolic metrics, closed curves and laminations.  Bonahon's key idea is to embed $\euC(S)$, 
$\T(S)$ and $\MF(S)$, into the space of geodesic currents $\CC(S)$.  Our next result extends
the theory to flat structures.

%%%%%%%ORIGINAL
\begin{theorem} \label{T:current embedding}
There is an embedding
\[ \Flat(S) \to \CC(S)\]
denoted by $q\mapsto L_q$ so that for $q\in\Flat(S)$ and $\alpha\in\euC$,
we have
$\I(L_q,\alpha)=\ell_q(\alpha)$.
Furthermore, after projectivizing, $\Flat(S)\to \PCC(S)$ is still an
embedding.
\end{theorem}

As a consequence, we obtain a continuous homogeneous extension of the flat length function in Corollary \ref{C:flat continuous}, 
\[ \Flat(S) \times \MF(S) \to \mathbb R, \]
making it meaningful to discuss the length of a foliation:

\begin{remark}
For the purpose of geodesic currents, punctured surfaces are treated as surfaces with holes; see Section \ref{preliminary section}.
\end{remark}

As $\PCC(S)$ is compact, Theorem \ref{T:current embedding} provides a compactification of $\Flat(S)$, 
and it is invariant under the action of
the mapping class group. Bonahon proved that for closed surfaces, the analogous compactification of $\T(S)$ is
precisely the Thurston compactification by projective measured laminations. For the compactification of $\Flat(S)$, we
also find a geometric interpretation of the boundary points as {\em mixed structures} on $S$.
A mixed structure is a hybrid of a flat structure on a subsurface (with boundary length zero) and a measured
lamination on the complementary subsurface.  We view the space of mixed structures as a subspace of $\CC(S)$, and
thus for any mixed structure $\eta$, there is a well-defined intersection number $\I(\eta, \cdot)$.
This theory is developed in Section \ref{boundaryflat}.

%%%%%%%ORIGINAL
\begin{theorem} \label{T:boundary}
The closure of $\Flat(S)$ in $\PCC(S)$ is exactly the space $\PMix(S)$.
That is,
for any sequence $\{q_n\}$ in $\Flat(S)$, after passing to a subsequence
if necessary,
there exists a mixed structure $\eta$ and a sequence of positive real
numbers
$\{t_n\}$ so that
$$
\lim_{n \to \infty} t_n \ell_{q_n}(\alpha) = \I(\alpha,\eta).
$$
for every $\alpha \in \euC$.  Moreover, every mixed structure is a limit
of a sequence
in $\Flat(S)$.
\end{theorem}

In Sections \ref{boundaryflat} and \ref{S:Finalremarks} we make several other comparisons between this compactification and the Thurston compactification of $\T(S)$.

%%%%%%%%%%%%%%%%%%%%%%%%%%%%
\subsection{Context: Other spectral rigidity results} \label{S:context}

Spectral rigidity of $\euS$ over $\T(S)$ was generalized considerably by Otal
\cite{otal}, who showed that $\euC$ is spectrally
rigid over $\calG_-(S)$, the space of all negatively
curved metrics on $S$ up to isotopy.  Hersonsky-Paulin
\cite{her-paul} generalized this further to show that $\euC$ is
spectrally rigid over negatively curved cone metrics.
This was pushed in a different direction by Croke \cite{croke},
Fathi \cite{fathi} and Croke-Fathi-Feldman
\cite{croke-ff} where it was shown that $\euC$ is spectrally rigid for
various qualities of nonpositively curved
metrics (for more precise statements, see the references).

While these results allow for rather general classes of metrics, the use
of all closed curves, not just the simple
ones, is essential.  Indeed, it follows from a result of Birman-Series
\cite{birmanseries} that, in general, we should
\textit{not} expect $\euS$ to be spectrally rigid for an arbitrary class
of negatively curved metrics, since simple closed curves miss most of the surface (see \S
\ref{S:Finalremarks}).

We saw above in Theorem \ref{T:rigid iff dense} that a set of curves must be dense in 
the sphere $\PMF$ in order to be spectrally rigid over $\Flat(S)$.
This stands in contrast with the situation for hyperbolic metrics, where it is known that there are finite spectrally 
rigid sets; in fact, $2\xi+1$ curves, one more than the dimension of $\T(S)$, are sufficient 
(see \cite{ham6g-5+2n, ham6g-5,schmutz}).  
In this regard, $\Flat(S)$ bears a resemblance 
to  Outer space, $\CV(F_n)$. The Culler--Vogtmann Outer space, built to study the 
group $Out(F_n)$ in analogy to the relationship between $\T(S)$ and the mapping class 
group, consists of
metric graphs $X$ equipped with a isomorphisms $F_n \to \pi_1(X)$ (under
the equivalence relation of graph isometries which respect the isomorphism up to 
conjugacy). Recycling notation suggestively, let $\euC$ denote the set of conjugacy 
classes of nontrivial elements of $F_n$. Given an element $X \in
\CV(F_n)$, and a conjugacy class $\alpha\in\euC$,  we write
$\ell_X(\alpha)$ for the minimal-length representative of $\alpha$ in $X$.
We can define a length spectrum just as above:
$$
\lambda_\Sigma(X) = (\ell_X(\alpha))_{\alpha \in \Sigma} \in \mathbb R^\Sigma
$$ 
for $X \in \CV(F_n)$ and $\Sigma\subset\euC$.
Accordingly, we say that $\Sigma$ is spectrally rigid over $\CV(F_n)$
if $X \mapsto \lambda_\Sigma(X)$ is injective.

The full set $\euC$ is spectrally rigid over $\CV(F_n)$
~\cite{alperin-bass,culler-morgan}. However, Smillie and Vogtmann
(expanding on a similar result of Cohen,
Lustig and Steiner ~\cite{c-l-s}) showed that no finite subset
$\Sigma\subset \euC$ is spectrally rigid over Outer
space (or even the reduced Outer space) by finding a $(2n-5)$-parameter
family of graphs over which
$\lambda_\Sigma$ is constant \cite{smillie-vogtmann}. Thus, Theorem
\ref{T:nonrigid} is the analog for $\Flat(S)$ of the Smillie--Vogtmann
result. Our proof of Theorem \ref{T:nonrigid} adapts the key idea from
Smillie--Vogtmann to surfaces by appealing to Thurston's theory of train
tracks; see \S \ref{S:families}.
This justifies the remark that from the point of view of length-spectral rigidity,
flat metrics might be said to resemble metric graphs more 
closely than hyperbolic metrics.

Finally, we briefly consider unmarked inverse spectral problems for the metrics in
$\Flat(S)$. Kac memorably asked in 1966 whether
one can ``hear the shape of a drum,'' or determine a planar region by
the eigenvalues of its Laplacian. Sunada's work
in the 1980s established a means of generating examples of hyperbolic
surfaces which are not only isospectral with
respect to their Laplacians, but iso-length-spectral as well. That is,
let the {\em unmarked length spectrum} be the
nondecreasing sequence of numbers
$$\Lambda_\euC(\rho)=\{\ell_\rho(\gamma_1)\le \ell_\rho(\gamma_2)\le
\cdots
\}_{\gamma_i\in\euC},$$
appearing as lengths of closed curves on $S$, listed with multiplicity.
Sunada's construction produces a supply
of examples of hyperbolic metrics $m,m'$ such that
$\Lambda_{\euC}(m)=\Lambda_{\euC}(m')$.
In \S~\ref{sunada}, we remark that the
Sunada construction carries over to our flat metrics in the same way.

\smallskip

\noindent
{\bf Acknowledgments.} This work began while all three authors were
visiting MSRI in Berkeley, CA for the Fall 2007 programs on
\textit{Teichm\"uller theory and Kleinian groups} and \textit{Geometric
group theory}.  We would like to thank MSRI and all of the participants
for providing such a stimulating mathematical environment.  We would
also like to thank Daryl Cooper for discussions of his work on
projective structures that suggested the idea of mixed structures.

%%%%%%%%%%%%%%%%%%%%%%%%%%%%%%%%%%%%%
\section{Preliminaries:  Flat structures and geodesic currents}
\label{preliminary section}
%%%%%%%%%%%%%%%%%%%%%%%%%%%%%

In this section, we will briefly describe the background and preliminary
material on \Teich theory, semi-translation surfaces, flat
metrics, and Bonahon's theory of geodesic currents.  We refer the
reader to  \cite{bonahon-bouts}, \cite{bon-currents}, \cite{gardiner-lakic}, \cite{pap-handbook}, and \cite{strebel}.

In what follows, $S$ is a finite-type surface.  That is, $S$ is obtained from a closed surface $\hat S$ by removing a finite set $P \subset \hat S$ of marked points.  The genus $g$ and number of punctures $n = |P|$ determine the topological complexity 
$$\xi=\xi(S) = 3g-3 + n.$$
Recall that \Teich space $\T(S)$, parametrizing the hyperbolic metrics on $S$ together with a marking of their
curves, is homeomorphic to a ball of dimension $2\xi$.

%%%%%%%%%%%%%%%%%%%%%%%%%%%%
\subsection{Quadratic differentials and semi-translation structures}
%%%%%%%%%%%%%%%%%%%%%%%%%%%%

By a \emph{quadratic differential} on $S$ we mean a complex structure on
$\hat S$ together with an integrable meromorphic quadratic differential.
The quadratic differential is allowed to have poles of degree one at
marked points and is assumed to be holomorphic on $S$. The space of all quadratic 
differentials, defined up to isotopy, is denoted $\calQ(S)$. A point of $\calQ(S)$ will 
be denoted $q$, with the underlying complex structure implicit in the notation.  
Reading off the complex structures, we obtain a projection to the \Teich space
\[ \pi: \calQ(S) \to \T(S).\]
This projection is canonically identified with the cotangent bundle to
$\T(S)$; hence $\calQ(S)$ has a real dimension of $4\xi$.

Integrating the square root of a nonzero quadratic differential $q$ in a
small neighborhood of a point where $q$ is
nonzero produces \textit{natural coordinates} $\zeta$ on $S$ in which $q
= d\zeta^2$.  The collection of all natural coordinates gives an atlas
on the complement of the zeros of $q$ for which the transition functions
are given by maps of the form $z \mapsto \pm z + c$ for $c \in \C$ (called
\textit{semi-translations}).
The Euclidean metric is preserved by these transition functions
and so pulls back to a Euclidean metric on the complement of the zeros
of $q$ in $S$.  The integrability of $q$ implies that the metric has
finite total area.

The completion of the metric is obtained by replacing the zeros of $q$
as well as the points $P$ to obtain
the surface $\hat S$.
If $q$ has a zero of order $p$ at one of the completion points, then
there is a cone singularity with cone angle $(2+p)\pi$ (here a pole at a
point of $P$ is thought of as a zero of order $-1$).  Thus the metric on $S$ 
is locally $\CAT(0)$ (or nonpostiively curved in the sense of comparison geometry), 
although the metric on $\hat S$
may not be. We also use $q$ to denote the completed metric on $\hat S$. 

A \textit{semi-translation structure} is a locally $\CAT(0)$ Euclidean
cone metric on $S$, whose completion is $\hat S$,
together with an atlas defining the metric away from the cone points for which the transition functions
are semi-translations.
The atlas determines (and is determined by) a preferred vertical direction.
Given a semi-translation structure, there is a unique complex structure and integrable holomorphic
quadratic
differential for which the charts in the atlas are natural coordinates.
This determines a bijection between the set of nonzero quadratic differentials and
the set of semi-translation structures on $S$, which we use to identify the two
spaces.
The \Teich metric is induced by the co-norm on $\calQ(S)$ which comes from the
area
of the associated semi-translation structure on $S$.  The unit cotangent
space, $\QQ(S)$,
is thus precisely the set of unit-area semi-translation structures on
$S$.

A semi-translation structure can also be described combinatorially as a
collection of (possibly punctured)
polygons in the Euclidean plane with sides identified in pairs by an isometry
which is the restriction of a semi-translation.

The group $\SL_2(\R)$ acts naturally
on the space of quadratic differentials by
$\R$--linear transformation on the natural coordinates.
The geodesics in the \Teich metric are precisely
projections to $\T(S)$ of orbits of the $\SL_2(\R)$ diagonal on an
initial quadratic differential $q_0$:
$$\gamma(t)=\left\{ \pi(A_t.q_0) : A_t=\matrix {e^t}00{e^{-t}} , t\in\R
\right\} .$$
The {\em Teichm\"uller disk} $\H_q$ of a quadratic differential $q$ is the
projection to $\T(S)$ of its entire $\SL_2(\R)$
orbit; it is an isometrically embedded copy of the hyperbolic plane of curvature $-4$.

We let $p:\widetilde S \to S$ denote the universal covering of $S$, with $\pi_1(S)$ acting by covering transformations.
The metric $q$ pulls back to a metric $\tilde q = p^*(q)$ on $\widetilde S$ which is again locally $\CAT(0)$.
When $S$ is a closed surface, $(\widetilde S,\tilde q)$ is a complete, geodesic $\CAT(0)$ space.
If $S$ has punctures, then $(\widetilde S,\tilde q)$ is incomplete, and we write $(\bar S,\tilde q)$ for the completion, obtaining a geodesic $\CAT(0)$ space.
The covering $p:\widetilde S \to S$ can be extended to the completions which we also 
denote by $p$.
This extension can be be viewed as a branched cover, infinitely branched over $P$, and we let $\widetilde P$ denote the preimage of $P$ in $\bar S$.

\subsection{Measured foliations and measured laminations}

We write $\MF=\MF(S)$ for the space of (measure classes
of) measured foliations on $S$, and 
$\PMF=\PMF(S)$ to denote projective measured foliations.  A curve
$\alpha \in \euS$ canonically determines a measured foliation
with all nonsingular leaves closed and homotopic to $\alpha$.  
We use this to view $\mathbb R_+ \times \euS$
and $\euS$ as subsets of $\MF$ and $\PMF$, respectively.
We also write
\[
\I:\MF \times \MF \to \mathbb R
\]
for Thurston's geometric intersection number.  This is the unique
homogeneous continuous extension of the usual
geometric intersection number on $\euS \times \euS$, via the
inclusions described above.

The vertical foliation for a nonzero quadratic differential $q\in
\calQ(S)$ is given by $| \Real(\sqrt{q}) |$. Let
$\foltheta$ be the foliation  $| \Real( e^{i\theta}\sqrt q ) |$ for
$\theta\in\RP^1$, so that the vertical foliation
of $q$  is $\nu_q:=\nu_q^0$. By setting
$$\MF(q):=\{   \  t\cdot\foltheta \   :  \  \theta\in \RP^1, t\in\R_+
\},$$ we obtain the set of all measured
foliations which are straight in some direction on $q$, with measure
proportional to Euclidean distance between leaves.
We write $\PMF(q)$ for the projectivization of $\MF(q)$.

It will be useful to pass back and forth between measured foliations and
measured laminations.  We denote the space of
measured laminations by $\ML$ and the projective measured laminations by
$\PML$.  We identify $\MF$ with  $\ML$ and
$\PMF$ with $\PML$ in the natural way extending the canonical
inclusions of $\euS$.  See \cite{levitt} for an
explicit procedure for constructing laminations from foliations.

\subsection{Flat structures} \label{S:flatstructures}

Quadratic differentials that represent the same metric differ only by a rotation.
Accordingly, the space of \textit{flat metrics} is defined as
$$\Flat(S)=\QQ(S) \Big/  {q\sim e^{i\theta}q}.$$
Equivalently, an element of $\Flat(S)$ is a Euclidean cone metric on $S$
which is locally $\CAT(0)$, with holonomy in $\{\pm
I\}$, completion $\hat S$, and total area one. This is almost identical
to the notion of a quadratic differential, but
there is one missing piece of data, namely the preferred vertical direction which
is determined by the atlas of natural
coordinates.  We write $q$ to denote a point in $\QQ(S)$ or the
associated equivalence class in $\Flat(S)$.  Note that
$\MF(q)$ and $\PMF(q)$ are well-defined for $q\in\Flat(S)$.  Also, each Teichm\"uller disk
$\H_q$ lifts to an embedded disk in $\Flat(S)$, and in fact,
$\Flat(S)$ is foliated by \Teich disks.

%%%%%%%%%%%%%%%%%%%%%
\subsection{Geodesics}

Let $q$ be a quadratic differential on $S$ and $(\bar S,\tilde q)$ the metric completion 
of the metric pulled back to the universal cover as described above. Every curve 
$\alpha \in \euC$ has a \emph{$q$--geodesic representative} in the following sense; 
for a map $\alpha \from S^1 \to S$ from the unit circle to $S$, there is an isometry 
$\tilde \alpha_q \from \R \to (\bar S,\tilde q)$ such that a subgroup of
$\pi_1(S)$ corresponding to the curve $\alpha$ preserves the image
$\tilde \alpha_q(\R)$.  The projection of this to $\hat S$ is the
$q$--geodesic representative of $\alpha$ and we denote it by $\alpha_q$.
(See ~\cite{rafi} for more details.) We call the isometry $\tilde \alpha_q$, or any 
$\pi_1(S)$--translate of it, a \emph{lift} of $\alpha_q$.

The geodesic representative of $\alpha$ is unique (up to parameterization), except when there are a 
family of parallel geodesic representatives foliating a flat cylinder. The geodesic 
representative of a simple closed curve need not be simple, and the geodesic representatives of different curves may not be different. For example, curves that 
go around a puncture different number of times can have the same geodesic 
representative that passes through the puncture (the number of times a curve
goes around the puncture is not detectable from the geodesic representative). 
However, for every curve $\alpha$, there is always a sequence of representatives 
of the homotopy class of $\alpha$ in $S$ converging uniformly to $\alpha_q$.

When $S$ is a punctured surface, we will also be interested in homotopy classes of essential proper paths in $S$.
These are paths $\alpha \from I \to \hat S$, defined on some closed interval $I$, for which the interior of $I$ is mapped to $S$ and the endpoints are mapped to $P$.
Here, two such paths are homotopic if there is a homotopy relative to the endpoints so that throughout the homotopy the interior of $I$ is mapped to $S$.
We denote the set of all homotopy classes of essential curves and paths by $\euC'(S)$, which is equal to $\euC(S)$ if $S$ is closed.
Every element of $\euC'(S)$ has a unique geodesic representative, which we view as 
the projection of an isometry $\tilde \alpha_q \from I \to (\bar S,\tilde q)$ to $\hat S$, and is again denoted by $\alpha_q$.
Again, $\alpha_q$ is a uniform limit of representatives of the homotopy class of $\alpha$.

When a curve $\alpha$ has non-unique geodesic representatives that foliate 
a cylinder, we say $\alpha$ is a \textit{cylinder curve} and we define the
\textit{cylinder set} of $q$, denoted by $\cyl(q)$, to be the set of all cylinder
curves with respect to $q$.

A \textit{saddle connection} is a geodesic segment whose endpoints are
(not necessarily distinct) singularities or points of $P$, and
which has no singularities in its interior.  When $\alpha \in \euC'(S)$ is not a cylinder curve,
the (unique) geodesic representative is made up of concatenations of saddle connections.  (In fact,
each boundary component of a cylinder
is a union of saddle connections, so even cylinder curves have
representatives of this form.)  If we write this
concatenation as
\[\alpha_q = \alpha^1\cdots\alpha^k, \]
and let $r_j$ denote the Euclidean length of  $\alpha^j$, then  $\ell_q(\alpha)$ is
just $r_1 + \cdots + r_k$.

If we view $q$ as a quadratic differential (and not just as a flat structure), then each $\alpha_j$
makes some angle $\theta_j$ with the horizontal
direction.

\begin{lemma} \label{L:length integral}
For all $q \in \QQ(S)$ and $\alpha \in \euC'(S)$, we have
\[ \ell_q(\alpha) = \frac{1}{2} \int_0^\pi \I(\foltheta,\alpha) d
\theta.\]
\end{lemma}
\begin{proof}
This is a computation:
\begin{align*}
\int_0^\pi \I(\foltheta,\alpha) \, d \theta
 & = \int_0^\pi \left( \sum_{j=1}^k \int_{\alpha_j} | \Real(
e^{i\theta}\sqrt q \, ) | \right) d\theta\\
 & = \sum_{j=1}^k \int_0^\pi r_j |
\cos( \theta + \theta_j ) | \, d \theta =  \sum_{j=1}^k 2r_j \, = \, 2
\ell_q(\alpha). \qedhere
\end{align*}
\end{proof}

While the $q$--geodesics $\alpha_q$ and $\beta_q$ are not necessarily
embedded or transverse, they do meet minimally in
a certain sense.  Namely, appealing to the $\CAT(0)$ structure, we first
note that any two lifts $\widetilde \alpha_q$
and $\widetilde \beta_q$ meet in a point, in a geodesic segment, or they
are disjoint.  If the endpoints at infinity
of $\widetilde \alpha_q$ and $\widetilde \beta_q$ nontrivially link,
then we call these intersections \textit{essential
intersections}. It follows that $\I(\alpha,\beta)$ is the number of
$\pi_1(S)$--orbits of essential intersections over
all lifts of $\alpha_q$ and $\beta_q$.

\begin{remark}\label{necsuffgeo}
We make an elementary but very useful observation that identifies the geodesics in a flat
metric $q$.  First consider the case that $S$ is closed.
Given a representative of $\alpha$ built as a concatenation of saddle connections $\alpha^1 \cdots \alpha^k$,
a necessary and sufficient condition for this to be a $q$-geodesic is that the angles
between successive $\alpha^i$ measure at least $\pi$ on both sides.
When $P$ is nonempty, we need to modify this slightly.
Suppose $\alpha^1 \cdots \alpha^k$ is a representative of $\alpha$ in $\hat S$ and consider a lift of this 
representative to $\bar S$; that is, $\alpha^1\cdots \alpha^k$ is a limit of representatives of $\alpha$ in $S$ and
the lift is a limit of lifts.  
Then an angle of at least  $\pi$ is subtended at each point in $\widetilde P$.
(Note that points of $\widetilde P$ are on the boundary, so there is a unique well-defined angle at each such 
point met by the lift.)
\end{remark}

%%%%%%%%%%%%%%%%%%%%%%%%%%%%%%%
\subsection{Geodesic currents}
%%%%%%%%%%%%%%%%%%%%%%%%%%%%%%%

For this discussion, we first restrict to the closed case ($P=\emptyset$).  We fix any
geodesic metric $g$ on $S$. We can pull back this metric to the
universal covering $p:\tilde S \to S$, so that the
covering group action of $\pi_1(S)$ on $\tilde S$ is by isometries.  We
let $\tilde S_\infty$ denote the Gromov
boundary of $\tilde S$, making $\tilde S \cup \tilde S_\infty$ into a
closed disk.  This compactification is
independent of the choice of metric (in the sense that a different
choice of metric gives an alternate compactification
for which the identity extends to a homeomorphism of the boundary
circles).

We consider the space
\[ \G(\tilde S) = (\tilde S_\infty \times \tilde S_\infty \setminus
\Delta) \Big/ (x,y) \sim (y,x). \]
With respect to our metric,
this is precisely the space of unoriented
bi-infinite geodesics in $\tilde S$ up to
bounded Hausdorff distance.  We endow $\G(\tilde S)$ with
the diagonal action of $\pi_1(S)$.

A \textit{geodesic current on $S$} is a $\pi_1(S)$--invariant Radon
measure on $\G(\tilde S)$.  The set of all geodesic
currents is made into a (metrizable) topological space by imposing the
weak* topology, and we denote this space
$\CC(S)$. The associated space of \textit{projective currents} is the
quotient of the space of nonzero currents by positive real scalar multiplication,
and we denote it $\PCC(S)$.

The simplest examples of geodesic currents are defined by closed curves
$\alpha \in \euC$ as follows. Given such a
curve $\alpha$, we first realize it by a geodesic representative (with
respect to our fixed metric). The preimage
$p^{-1}(\alpha)$ in $\tilde S$ determines a discrete subset of $\G(\tilde S)$
(independent of the metric), and to this we can
associate a Dirac measure on $\G(\tilde S)$, for which
$\pi_1(S)$--invariance follows from the invariance of $p^{-1}(\alpha)$.
This injects
the set $\euC$ into $\CC(S)$, and we will thus view $\euC$ as a
subset of $\CC(S)$ when convenient.  While these
are very special types of geodesic currents, the set of positive real multiples
of all curves is in fact dense in $\CC(S)$, as shown in 
\cite{bon-currents}.

In \cite{bonahon-bouts}, Bonahon constructs a continuous extension for the intersection number to all currents.
\begin{theorem}[Bonahon] \label{T:i cont} The geometric intersection
number $\I:\euC(S) \times \euC(S) \to \R$ has a
continuous, bilinear extension
\[\I: \CC(S) \times \CC(S) \to \R.\]
\end{theorem}

Moreover, in \cite{otal}, Otal proved that $\I$ and $\euC$ can be used to
separate points:
\begin{theorem}[Otal] \label{T:otal}
Given $\mu_1,\mu_2 \in \CC(S)$, $\mu_1 = \mu_2$ if and only if
$\I(\mu_1,\alpha) = \I(\mu_2,\alpha)$ for all $\alpha
\in \euC$.
\end{theorem}

From this, one can easily deduce a convergence criterion, and also
define a metric on the space of currents which will be convenient for our purposes.
\begin{theorem} \label{T:metric}
A sequence $\mu_k \in \CC(S)$ converges to $\mu \in \CC(S)$ if and
only if 
$$
\lim\limits_{k \to \infty} \I(\mu_k,\alpha) = \I(\mu,\alpha),
$$ 
for all $\alpha \in \euC.$ Furthermore, there exist
$t_\alpha \in \R_+$ for each $\alpha\in\euC$ so that
\[
d(\mu_1,\mu_2) = \sum_{\alpha \in \euC} t_\alpha 
\big| \I(\mu_1,\alpha) - \I(\mu_2,\alpha) \big|
\] defines a proper metric on $\CC(S)$ which is
compatible with the weak* topology.
\end{theorem}

Before we prove this theorem, we recall one further fact due to Bonahon
\cite{bon-currents} which we will need.  We say that a geodesic current $\nu$ is \textit{binding}
if for every $(x,y) \in \G(\tilde S)$, there is an $(x',y')$
in the support of $\nu$ such that $(x,y)$ and $(x',y')$ link in $\tilde
S_\infty$.  With respect to any fixed metric,
this is equivalent to requiring that every bi-infinite geodesic in $\tilde S$
intersects some geodesic in the support of $\nu$. It
follows, as discussed by Bonahon, that
any binding current and any nonzero current have positive intersection number.
As an example, any filling curve or union of curves
determines a binding current.

\begin{proposition}[Bonahon] \label{P:compact currents} If $\nu$ is a
binding geodesic current and $R > 0$, then the set
\[ \{ \mu \in \CC(S) \, | \, \I(\mu,\nu) \leq R \} \]
is a compact set.  Consequently, the set
\[ \left\{ \frac{\mu}{\I(\mu,\nu)} \, \Big| \, \mu \in \CC(S) \setminus \{0
\} \right\} \]
is compact, and hence so is $\PCC(S)$.
\end{proposition}

\begin{proof}[Proof of Theorem \ref{T:metric}]
Continuity of $\I$ implies $\I(\mu_k,\alpha) \to \I(\mu,\alpha)$ for all $\alpha \in \euC$ if
$\mu_k \to \mu$.   To prove the other direction,
assume $\I(\mu_k,\alpha) \to \I(\mu,\alpha)$ for all $\alpha \in \euC$.  In particular, if we let
$\alpha_0 \in \euC$ be a filling curve (so the
associated current is binding), then
$\I(\mu_k,\alpha_0),\I(\mu,\alpha_0) \leq R$ for some $R
> 0$.  So, $\{\mu_k\} \cup \{\mu\}$ is contained in some compact set
by Proposition \ref{P:compact currents}.

Since $\CC(S)$ is metrizable, it follows that there is a convergent
subsequence $\mu_{k_n} \to \mu'$ for some $\mu' \in \CC(S)$. Continuity of
$\I$ implies that $\I(\mu,\alpha) = \I(\mu',\alpha)$ for all $\alpha$,
and so Theorem \ref{T:otal} guarantees that $\mu
= \mu'$. Since this is true for any convergent subsequence of
$\{\mu_k\}$ it follows that $\mu_k \to \mu$.  This completes the proof of the first statement of the theorem.

To build the metric we must first find the numbers $\{t_\alpha\}$.  For
this, we observe that for any $\mu \in \CC(S)$ and fixed choice of a
filling curve $\alpha_0$,
the numbers
\[ \left\{ \frac{\I(\mu,\alpha)}{\I(\alpha_0,\alpha)} \right\}_{\alpha
\in \euC}  =
 \left\{ \I \left( \mu, \ \frac{\alpha}{\I(\alpha_0,\alpha)}\right) \right\}_{\alpha \in \euC}
\] are uniformly bounded.  This follows from the fact that the set of
currents
\[
\left\{ \frac{\alpha}{\I(\alpha_0,\alpha)} \right\}_{\alpha \in \euC}
\]
is precompact by Proposition \ref{P:compact currents}.

Now we enumerate all closed curves $\alpha_0,\alpha_1,\alpha_2,... \in
\euC$ ($\alpha_0$ still denoting our filling
curve).  Set $t_k = t_{\alpha_k} = 1/(2^k \I(\alpha_0,\alpha_k))$.  It
follows that
\[ 
\sum_{k=0}^\infty t_k \I(\mu,\alpha_k) = \sum_{k=0}^\infty \frac{1}{2^k} \, 
\I \left(\mu,\frac{\alpha_k}{\I(\alpha_0,\alpha_k)} \right)
\]
converges and hence the series for $d$ given in the statement of the
proposition converges.  Symmetry and the triangle
inequality are immediate, and positivity follows from Theorem
\ref{T:otal}.  The fact that the topology agrees with the
weak* topology is a consequence of the first part of the Theorem and the
fact that $\CC(S)$ is metrizable (hence first
countable, so determined by its convergent sequences).

Finally, we verify that the metric is proper.
Proposition \ref{P:compact currents} implies that for any binding
current $\nu \in \CC(S)$, the set
\[ A =  \left\{ \frac{\mu}{\I(\mu,\nu)} \, \Big| \, \mu \in \CC(S)
\setminus \{0 \} \right\} \]
is compact.  Since $d$ is continuous, the distance from $0$ to any point
of $A$ is bounded above by some $R > 0$
and below by some $r > 0$.
Furthermore, for any $\mu \in \CC(S)$ and $t \in \mathbb R_+$, we have
\[d(t\mu,0) = t\cdot d(\mu,0).\]
Hence, the compact set
\[ A' = \{ t \mu \, | \, \mu \in A \, , \, t \in [0,1] \} \]
is contained in the ball of radius $R$ and contains the ball of radius
$r$.  From this and the preceding equation, it
follows that for any $\rho
>0$, the closed ball of radius $\rho > 0$ about $0$ is a compact set.
That is, $d$ is a proper metric.
\end{proof}

\subsection{Punctured surfaces}

The situation for punctured surfaces requires more care.  First, we replace all punctures by holes, so that we may
uniformize $S$ by a convex cocompact hyperbolic surface.   That is, we give $S$ a complete hyperbolic metric (of
infinite area) so that $S$ contains a compact, \textit{convex core} which we denote $\core(S)$.  To describe $\core(S)$
concretely, first consider the universal covering $\tilde S \to S$ (with $\tilde S$ isometric to the hyperbolic plane) together
with the isometric action of $\pi_1(S)$ by covering transformations.  We denote the limit set of the action on the
circle at infinity of $\tilde S$ by $\Lambda \subset \tilde S_\infty$.  The convex hull of $\Lambda$ in $\tilde S$ is a
closed, $\pi_1(S)$--invariant set which we denote $\hull(\Lambda)$, and the quotient by $\pi_1(S)$ is precisely $\core(S)$.
The inclusion $\core(S) \subset S$ is a homotopy equivalence and the convex cocompactness means that
$\core(S)$ is compact. Let $G(\hull(\Lambda))$ denote the space of geodesics in $\widetilde S$ with both endpoints
in $\Lambda$.  Thus,
$$G(\hull(\Lambda)) \cong (\Lambda \times \Lambda - \Delta) / (x,y) \sim (y,x).$$

A geodesic current on $S$ is now defined to be a $\pi_1(S)$--invariant Radon measure on $G(\hull(\Lambda))$.  Equivalently, we are considering
$\pi_1(S)$--invariant measures on $G(\tilde S)$ for which the support consists of geodesics that project entirely into $\core(S)$.  We use the same notation as before and denote the space of currents on $S$ by $\CC(S)$, endowed with the weak* topology.  Bonahon also proves that the associated projective space $\PCC(S)$ is compact and that the geometric intersection number on closed curves extends continuously to a symmetric bilinear function
\[ \I: \CC(S) \times \CC(S) \to \R. \]

In this setting, the conclusion of Theorem \ref{T:otal} is not true: the geodesic currents associated to boundary
curves have zero intersection number with every geodesic current.  We remedy this as follows.

First suppose that $\alpha: \R \to S$ is a proper bi-infinite geodesic.  If we let $\tilde \alpha:\R \to
\widetilde S$ denote a lift of $\alpha$, then both endpoints limit to points in $\widetilde S_\infty - \Lambda$. As
such, the set of all geodesics in $G(\hull(\Lambda))$ which transversely intersect $\tilde \alpha(\R)$ is a
compact set which we denote $A_{\tilde \alpha}$.  Given $\mu \in \CC(S)$, we define
\[ \I(\mu,\alpha) = \mu(A_{\tilde \alpha}).\]

\begin{lemma} \label{L:proper intersection}
For any proper bi-infinite geodesic $\alpha:\R \to S$, the function
\[\CC(S) \to \R\]
given by $\mu \mapsto \I(\mu,\alpha)$ is continuous and depends only on the proper homotopy class of $\alpha \in \euC'(S)$.
\end{lemma}
\begin{proof}
The $\pi_1(S)$--equivariance of $\mu$ shows that $\I(\mu,\alpha)$ is independent of the chosen lift $\tilde
\alpha:\R \to \widetilde S$.  Moreover, a proper homotopy $\alpha_t$ of $\alpha$ lifts to a homotopy $\tilde
\alpha_t$ for which no endpoint ever meets $\Lambda$.  It follows that $A_{\tilde \alpha_t} = A_{\tilde \alpha}$ for
all $t$ and so $\I(\mu,\alpha)$ depends only on the homotopy class $\alpha \in \euC'(S)$.

All that remains to prove is continuity.  Suppose $\mu_k \to \mu$ in $\CC(S)$.  Then since the characteristic function
$\chi$ of $A_{\tilde \alpha}$ is a compactly supported continuous function, it follows that
\[ \I(\mu_k,\alpha) = \int_{G(\hull(\Lambda))} \chi d \mu_k \to \int_{G(\hull(\Lambda))} \chi  d \mu =
\I(\mu,\alpha)\] as required.
\end{proof}

Appealing to the closed case, this provides us with enough intersection numbers to prove the analog of Theorem \ref{T:otal} in the present setting.

Let $DS$ be the double of $\core(S)$ over its boundary, which naturally inherits a hyperbolic metric from 
$\core(S)$.  We consider $\core(S)$ as isometrically embedded in $DS$.  The cover of $DS$ associated to 
$\pi_1(\core(S)) < \pi_1(DS)$ is canonically isometric to $S$, and we can identify the two surfaces, writing 
$S \to DS$ for this cover.  
Thus we have a canonical
identification of universal covers $\tilde S = \widetilde{DS}$.  The action of $\pi_1(S)$ on $\tilde S_\infty$ is
the restriction to $\pi_1(S) < \pi_1(DS)$ of the action of $\pi_1(DS)$.  Any geodesic current $\mu \in \CC(S)$
can be extended to a current in $\CC(DS)$, which we also denote $\mu$, by pushing the measure around via coset representatives of $\pi_1(S) < \pi_1(DS)$, making it $\pi_1(DS)$--equivariant.

This defines an injection $\CC(S) \to \CC(DS)$, and it is straightforward to check that 
this is an embedding.  It follows from Bonahon's construction of the intersection number 
function that $\I$ on $\CC(S)$ is just the restriction, via this embedding, of $\I$ on 
$\CC(DS)$.  If $\alpha$ is any closed geodesic on $DS$, then there are a finite 
(possibly zero) number of lifts of $\alpha$ to the cover $S \to DS$ that nontrivially meet 
$\core(S)$, and we denote these 
$$\alpha^1,\cdots,\alpha^k:\mathbb R \to S.$$
If the image is entirely contained in $\core(S)$, then there is only one lift, and it covers 
a closed geodesic. Otherwise, $\alpha^1,\cdots,\alpha^k$ is a union of proper geodesics 
in $S$.  An inspection of Bonahon's definition of $\I$ reveals that for any $\mu \in \CC(S)$,
\[\I(\mu,\alpha) = \sum_{i=1}^k \I(\mu,\alpha^i).\]
We can now prove the required analog of Theorem \ref{T:otal}.

\begin{theorem} \label{T:like otal}
Given $\mu_1,\mu_2 \in \CC(S)$, $\mu_1 = \mu_2$ if and only if $\I(\mu_1,\alpha) = \I(\mu_2,\alpha)$ for all $\alpha
\in \euC'(S)$.
\end{theorem}
\begin{proof}
If $\mu_1 \neq \mu_2$, we must find $\alpha \in \euC'(S)$ so that 
$\I(\mu_1,\alpha) \neq \I(\mu_2,\alpha)$. By Theorem \ref{T:otal}, there exists 
$\alpha \in \euC(DS)$ so that $\I(\mu_1,\alpha) \neq \I(\mu_2,\alpha)$. If
$\alpha$ is contained in $\core(S)$, then $\alpha \in \euC(S) \subset \euC'(S)$ and we 
are done. Otherwise, let $\alpha^1,...,\alpha^k \in \euC'(S)$ be the lifts as described 
above. Then
\[ \sum_{i=1}^k \I(\mu_1,\alpha^i) = \I(\mu_1,\alpha) \neq \I(\mu_2,\alpha) = \sum_{i=1}^k \I(\mu_2,\alpha^i).\]
But then $\I(\mu_1,\alpha^i) \neq \I(\mu_2,\alpha^i)$ for some $i$, completing the proof.
\end{proof}

We also easily obtain a version of Theorem \ref{T:metric}.
\begin{theorem} \label{T:metric2}
A sequence $\{\mu_k\} \in \CC(S)$ converges to $\mu \in \CC(S)$ if and only if 
$$
\lim\limits_{k \to \infty} \I(\mu_k,\alpha) = \I(\mu,\alpha)$$ 
for all $\alpha \in \euC'(S)$.  Furthermore, there exist $t_\alpha \in \R_+$ for
each $\alpha\in\euC'(S)$ so that
\[
d(\mu_1,\mu_2) = \sum_{\alpha \in \euC'(S)} t_\alpha 
\big| \I(\mu_1,\alpha)- \I(\mu_2,\alpha) \big|
\] 
defines a proper metric on $\CC(S)$ which is compatible with the weak* topology.\qed
\end{theorem}
\begin{proof}
Although we do not have Proposition
\ref{P:compact currents} over $S$, 
this proposition applied to $DS$ implies that if $\alpha_0 \in \euC(DS)$ is a filling curve,
then the associated proper geodesics $\alpha^1,...,\alpha^k \in \euC'(S)$ have the property that
\[ A = \left\{ \frac{\mu}{\sum_j \I(\mu,\alpha^j)} \, \Big| \, \mu \in \CC(S) \setminus 0 \right\} \]
is compact.  The proof continues as for Theorem \ref{T:metric}.
\end{proof}

%%%%%%%%%%%%%%%%%%%%%%%%%%%%%%%%%%%%%%%%%%%%%%%%%%%%%%%
\section{Spectral rigidity for simple closed curves} \label{S:all S rigid}
%%%%%%%%%%%%%%%%%%%%%%%%%%%%%%%%%%%%%%%%%%%%%%%%%%%%%%%

This section is devoted to the proof of Theorem~\ref{T:S rigid}.
We begin by considering the case of the torus.  This is not a step in
proving the theorem,
but the proof illustrates
a useful principle used later, and also shows that Theorem \ref{T:rigid iff dense} is false for tori 
(and similarly for once-punctured tori and four-times-punctured spheres).
\begin{proposition} \label{P:three curves torus}
The lengths of any three distinct primitive closed curves determine a flat metric on the torus.
\label{P:flattorus}
\end{proposition}

\begin{proof}
The \Teich space of unit-area flat tori is the hyperbolic plane.
Within this parameter space, prescribing the length of a given curve
picks out
a horocycle in $\H$. The intersection of two horocycles is at most two
points, so by
choosing three arbitrary curves, we can determine the flat metric
on a torus by their lengths.
\end{proof}

The proof of spectral rigidity for simple closed curves follows from a
series of lemmas.
The first states that $\lambda_\euS(q)$
determines $\cyl(q)$.

\begin{lemma} \label{L:spec gives cyl}
For $\alpha \in \euS$ and $q \in \Flat(S)$, $\alpha \not\in \cyl(q)$ if
and only if there exists $\beta \in \euS$ with
$\I(\alpha,\beta) \neq 0$ so that the following condition holds:
\begin{equation}\label{E:not cyl iff equal}
\ell_q(T_\alpha(\beta))-\ell_q(\beta)=\ell_q(\alpha) \cdot
\I(\alpha,\beta).\end{equation}
\end{lemma}
\begin{proof}
First, suppose $\alpha \in \cyl(q)$.  Fix any $\beta$ with
$\I(\alpha,\beta) \neq 0$. We must show that
$\alpha,\beta,q$ do not satisfy (\ref{E:not cyl iff equal}).

Let $\alpha_q$ denote a $q$--geodesic representative contained in the
interior of its Euclidean cylinder neighborhood
$C$ and let $\beta_q$ denote a $q$--geodesic representative of $\beta$.
Either $\beta_q$ is obtained by traversing a finite
number of saddle connections or else is itself a cylinder curve
(defining a different cylinder than $\alpha$) and
contains no singularities.  It follows that $\beta_q \cap C$ consists of
finitely many straight arcs connecting one
boundary component of $C$ to the other and the number of transverse
intersections of $\alpha_q$ and $\beta_q$ is
 $\I(\alpha,\beta)$.

We can construct a representative of $T_\alpha(\beta)$ as follows.  An
arc $\delta$ of the intersection $\delta \subset
\beta_q \cap C$ is cut by $\alpha_q$ into two arcs $\delta = \delta_0
\cup \delta_1$. To obtain $T_\alpha(\beta)$,
surger in a copy of $\alpha_q$ traversed positively; see
Figure \ref{F:surgery}. Observe that this is
necessarily not a geodesic representative since it makes an angle less
than $\pi$ at each of the surgery points.

\begin{figure}[ht]
\setlength{\unitlength}{0.01\linewidth}
\begin{picture}(80,29)
\put(0,0){\includegraphics[width=80\unitlength]{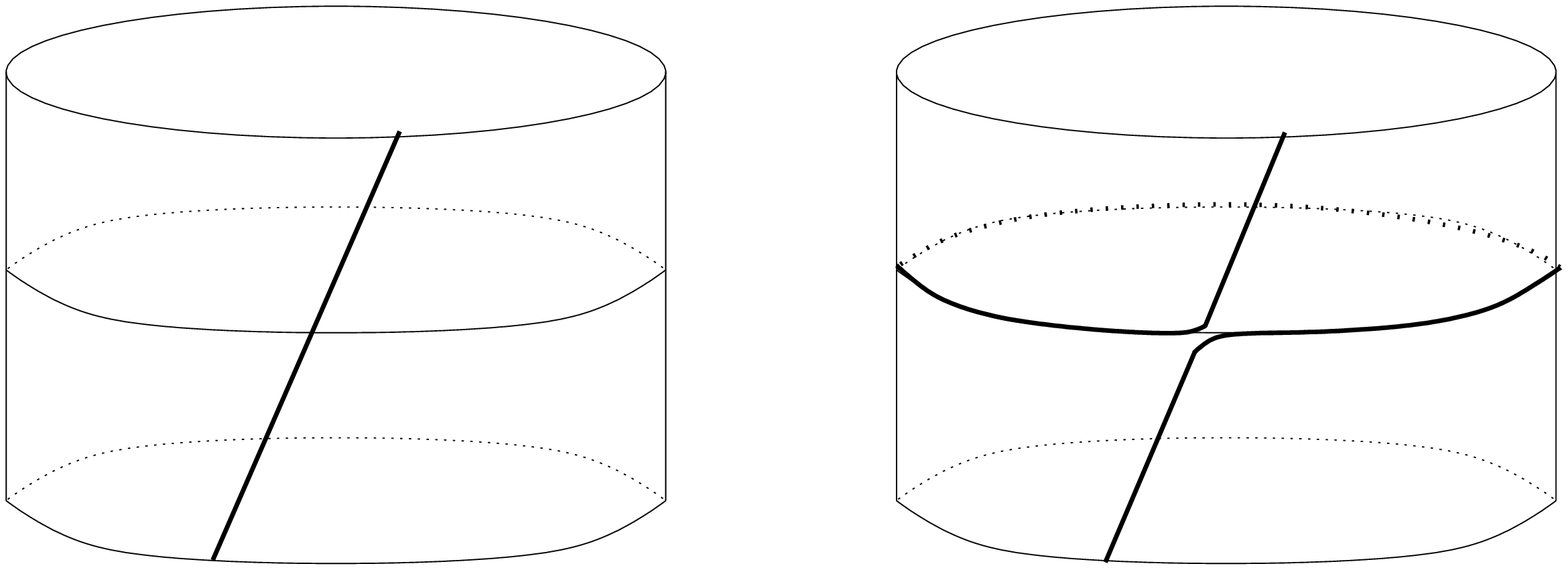}}
   \put(14,15){$\delta$}
    \put(55,15){$T_\alpha(\delta)$}
    \put(30.5,7){$\longleftarrow \quad C \quad \longrightarrow$}
\end{picture}
 \caption{A representative of the image of an arc $\delta$ under
$T_\alpha$} \label{F:surgery}
\end{figure}

Because $\alpha_q$ and $\beta_q$ are transverse, the number
$\I(\alpha,\beta)$ counts the number of intersection points
of $\alpha_q$ and $\beta_q$ which in turn counts the number of arcs
$\delta$ of intersection that $\beta_q$ makes with $C$. The
length of the representative $T_\alpha(\beta)$ we have constructed is
thus precisely
$$\ell_q(\beta)  + \ell_q(\alpha) \cdot \I(\alpha,\beta).$$
As we noted above, our representative is necessarily not geodesic, and hence
$$\ell_q(T_\alpha(\beta)) < \ell_q(\beta) + \ell_q(\alpha) \cdot
\I(\alpha,\beta).$$
Therefore (\ref{E:not cyl iff equal}) is not satisfied, proving the
first half of the lemma, since $\beta$ was arbitrary.\medskip 

We now assume $\alpha \not\in \cyl(q)$, and find $\beta$ with $\I(\alpha,\beta) \neq 0$ so that (\ref{E:not cyl iff
equal}) is satisfied. Assume for simplicity that $S$ is closed (the punctured case is similar).
Consider the universal cover $\widetilde S$ of $S$ equipped with the lifted metric of $q$, and fix
a lift $\widetilde \alpha_q$ of $\alpha_q$.
The bi-infinite geodesic $\widetilde \alpha$ separates
$\widetilde S$ into two components, $H^+ \cup H^-$. Let $h$ be an element of $\pi_1(S)$ that generates the stabilizer of
$\widetilde \alpha_q$, so that its action is by translation along $\widetilde \alpha_q$.

Because $\alpha$ is not a cylinder curve, $\widetilde \alpha_q$ is a
concatenation of saddle connections
meeting at singularities of $\widetilde q$.  
Consider the angles made on each of the two sides at the singularities.
If the angles were always $\pi$ on one side, then there is a parallel curve on $S$ that is
nonsingular, which means $\alpha$ itself is in $\cyl(q)$, contrary to assumption.
Thus, there is a
singularity $x^+$ so that the angle at $x^+$ on
the $H^+$ side made by the saddle connections meeting there is strictly
greater than
$\pi$, and likewise there is $x^-$ chosen relative to $H^-$.  

\begin{figure}[ht]
\setlength{\unitlength}{0.01\linewidth}
\begin{picture}(45,45)
\put(0,0){\includegraphics[width=45\unitlength]{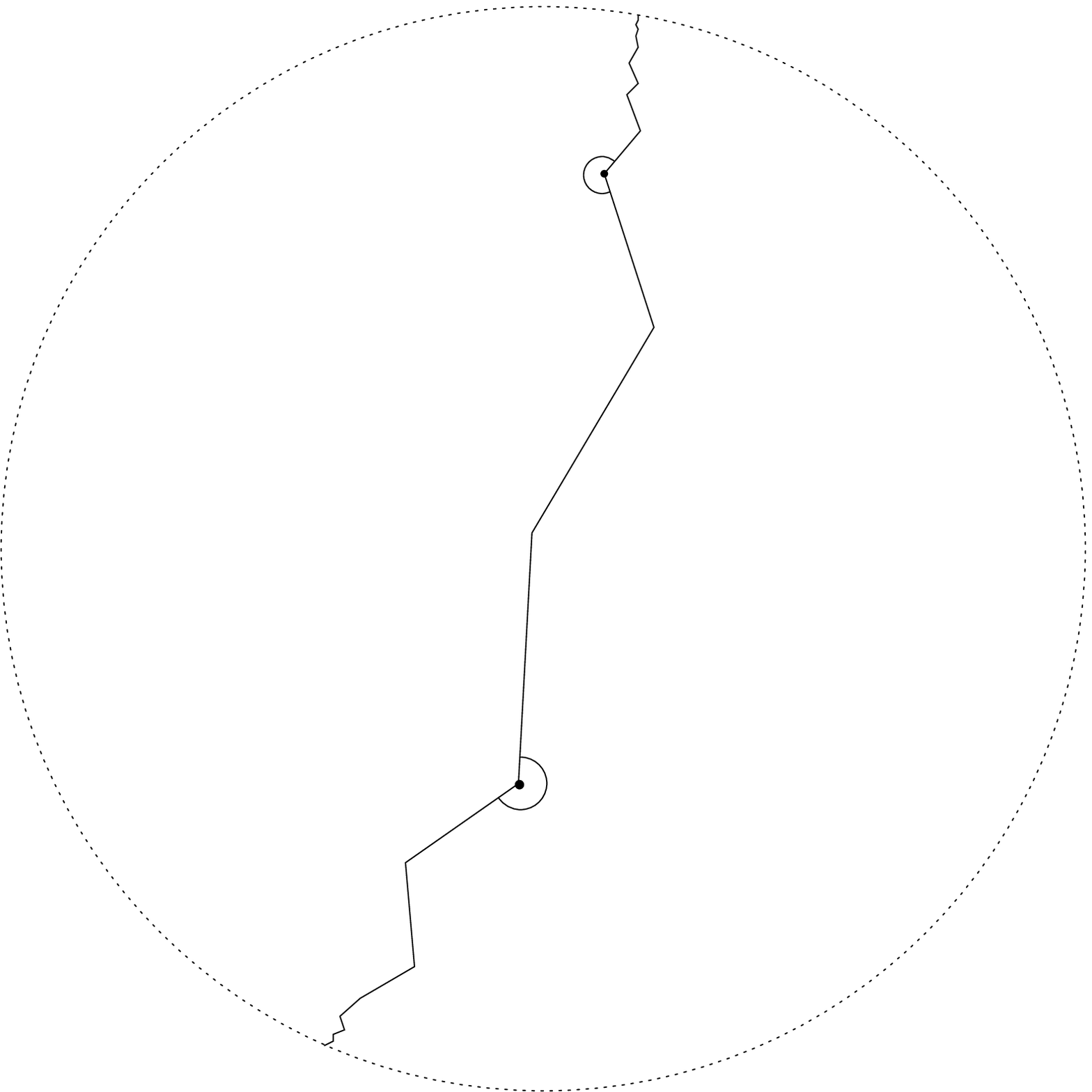}}
   \put(0,40){$\widetilde S$}
   \put(6,23){$H^-$}
   \put(35,19){$H^+$}
   \put(19,25){$\widetilde \alpha$}
   \put(28, 37){$x^-$}
   \put(18,37){$\pi < $}
   \put(17,13){$x^+$}
   \put(24,12){$> \pi$}
\end{picture}
 \caption{The lift $\widetilde \alpha$ and the singularities $x^\pm$.}
\label{F:goodsings}
\end{figure}

We choose geodesics $\gamma^\pm$ contained in $H^\pm$ meeting
$\widetilde \alpha_q$ precisely in the points $x^\pm$. Let
$A^+$ (respectively $A^-$) be the region on the circle at infinity bounded
by an endpoint of $\gamma^+$ (respectively
$\gamma^-$) and an endpoint of $\widetilde \alpha_q$, as shown in Figure
\ref{F:usesings}.

\begin{figure}[ht]
\setlength{\unitlength}{0.01\linewidth}
\begin{picture}(45,53)
\put(0,3){\includegraphics[width=45\unitlength]{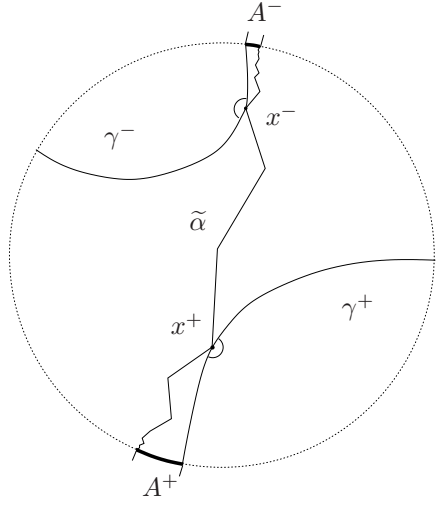}}
   \put(25,51){$A^-$}
   \put(14,1){$A^+$}
   \put(19,29){$\widetilde \alpha$}
   \put(17,18){$x^+$}
   \put(27,40){$x^-$}
\put(35, 20){$\gamma^+$}
\put(10,38){$\gamma^-$}
\end{picture}
\caption{The intervals along the boundary.}
 \label{F:usesings}
\end{figure}

Let $\beta_0 \in \euS$ be any curve with $\I(\alpha,\beta_0) = k \neq
0$.  For
each $1\le j \le k$, we pick a lift $\widetilde
\beta^j$
whose endpoints link those of $\widetilde \alpha$. By replacing
$\beta_0$
with its Dehn-twisted image
$\beta=T_\alpha^N(\beta_0)$ for large enough $N$, we can choose
$\widetilde
\beta^j$ so that
it has one endpoint in $A^+$ and the other in $A^-$, since the effect of
$T_\alpha$
is to shear along $\widetilde \alpha$.
Observe that
each such $\widetilde \beta^j$ includes the $\tilde q$--geodesic segment
$[x^+, x^-]$.
Therefore, for each of the $k$ essential intersections of $\beta$ with $\alpha$, 
the curve $\beta$ traverses some definite length of $\alpha$.
It follows that the geodesic representative of $T_\alpha(\beta)$ is
now exactly obtained from $\beta$ by surgering
in $k$ copies of $\alpha$.
From this, we get  \eqref{E:not cyl iff equal}, as required.
\end{proof}

\begin{corollary}
If $q,q' \in \Flat(S)$ and $\lambda_\euS(q) = \lambda_\euS(q')$, then
$\cyl(q) = \cyl(q')$.  \qed
\end{corollary}

The next lemma, combined with Lemma \ref{L:spec gives cyl}, reduces to a disk the
subspace of $\Flat(S)$ having prescribed
lengths.

\begin{lemma} \label{L:cyl gives H}
If $\cyl(q) = \cyl(q')$, then $\H_q = \H_{q'}$.
\end{lemma}

\begin{proof}
Suppose $\cyl(q)=\cyl(q')$.  First lift $q$ and $q'$ to arbitrary representatives 
in $\QQ$, also called $q$ and $q'$, so that it is well-defined to talk about particular directions.
Note that a cylinder curve, since it belongs a parallel family of nonsingular representatives, has 
a well-defined direction $\theta\in\RP^1$.
Next, recall that for any quadratic differential, the set of directions with at least
one cylinder is dense in $\RP^1$ by a result of Masur ~\cite{masur-cyl}.
Thus, for every uniquely ergodic foliation $\nu^\theta_q \in \PMF(q)$, there is a sequence of cylinder curves 
$\alpha_i \in \cyl(q)$ for which the directions converge: 
$\theta_i \to \theta$. It follows that 
$$
\nu^{\theta_i}_q \to \nu^\theta_q \qquad\text{as}\qquad i \to \infty.
$$  
Since $\I(\nu^{\theta_i}_q,\alpha_i) = 0$, it
follows that in $\PMF$, up to subsequence, we have $\alpha_i \to \mu \in \PMF$ with $\I(\mu,\nu^\theta_q) = 0$.
Since $\nu^{\theta}_q$ is uniquely ergodic, this means that $\mu$ and $\nu^\theta_q$ are equal, 
and hence $\alpha_i \to \nu^\theta_q$ in $\PMF$. 
From the assumption that $\cyl(q')=\cyl(q)$, it follows that $\nu$ is also in $\PMF(q')$. 
Thus the sets of uniquely ergodic
foliations in $\PMF(q)$ and $\PMF(q')$ are identical.

Consider a pair of uniquely ergodic foliations $\mu_0$ and $\nu_0$ in
$\PMF(q) \cap \PMF(q')$. There is a matrix $M$ (respectively, $M'$) in
$SL_2(\R)$
so that $\mu_0$ and $\nu_0$ are the vertical and the horizontal
foliations
of $M q$ (respectively, $M'q'$). However, there is a unique \Teich geodesic
connecting $\mu_0$ and $\nu_0$ (\cite{Gardiner-Masur}). Therefore, there is a time $t$ for which
$$
M' q' = A_t M q\qquad\text{for}\qquad A_t =\matrix{e^t}00{e^{-t}}.
$$ 
That is, $q'$ is in the $SL(2,\R)$ orbit of $q$, and hence $\H_q=\H_{q'}$.
\end{proof}

\begin{proof}[Proof of Theorem \ref{T:S rigid}]
Suppose $\lambda_{\euS}(q) = \lambda_{\euS}(q')$.  By Lemma \ref{L:spec
gives cyl}, $\cyl(q) = \cyl(q')$ and so Lemma
\ref{L:cyl gives H} implies $\H_q = \H_{q'}$.   A level set of the
length of a given cylinder curve on $\H_q = \H_{q'}$ is a
horocycle.  So if $\alpha,\beta,\gamma \in \cyl(q) = \cyl(q')$ have
distinct directions, then $q$ and $q'$ are
contained in the intersection of the same three distinct horocycles.  As
in the case of flat tori (Proposition
\ref{P:flattorus}), this implies $q = q'$.
\end{proof}

%%%%%%%%%%%%%%%%%%%%%%%%%%%%%%%%%%%%%
\section{Iso-length-spectral families} \label{S:families}
%%%%%%%%%%%%%%%%%%%%%%%%%%%%%%%%%%%%%

Here we show constructively that for a set of curves to be spectrally
rigid, its projectivization must not miss
any open set of $\PMF$.

%%%%%%%%REPEAT
\begin{nonrigid}
Suppose $\xi(S) \geq 2$.  
If $\Sigma \subset \euS \subset \PMF$ and $\overline \Sigma \neq \PMF$, then there is a 
deformation family $\Omega_\Sigma \subset \Flat(S)$ for which 
$\Omega_\Sigma \to \mathbb R^\Sigma$ is constant, and such that
the dimension of $\Omega_\Sigma$ is proportional to the dimension of $\Flat(S)$ itself.
\end{nonrigid}

In particular, no finite set of curves determines a flat metric. 
We will build deformation families of flat metrics in this section
 based on a train track argument.
We refer the reader to \cite{penner-harer} for a detailed discussion of train tracks.

Given a metric $\rho$ on $S$ (with metric completion $\hat S$), we call a train track $\tau \subset S$ {\em magnetic} with respect to 
$\rho$ if there exists a map $f:(\hat S,P) \to (\hat S,P)$, homotopic to the identity rel $P$, such that 
if $\gamma\subset \tau$ is a curve carried by $\tau$, then $f(\gamma)$ is a $\rho$--geodesic 
representative of $\gamma$ (up to parametrization).  
The magnetizing map $f$ should be thought of as taking a smooth realization of the train 
track to a geodesic realization
(compare Figure~\ref{smoothing} below).
In the examples in this section,  $f$ is a homeomorphism isotopic to the identity. 
More complicated maps $f$ are used to deal with the case of punctures, as presented in the 
appendix.

Informally, a train track is magnetic if geodesics ``stick to it'':  geodesics 
carried by $\tau$ actually live inside of the one-complex $f(\tau)$
as concatenations of the branches.  
Note that while magnetic train tracks
are easily constructed for flat metrics, they do not exist for any hyperbolic metric (or in fact for any 
complete Riemannian metric).

The strategy for proving Theorem \ref{T:nonrigid} is to first construct an initial train track $\tau$ on $S$ and 
a deformation family $\Omega \subset \Flat(S)$ so that $\tau$ is magnetic in $q$ for all 
$q \in \Omega$ and so that the length of any curve $\gamma$ carried by $\tau$ is constant on 
$\Omega$.  The train track $\tau$ we construct is complete and recurrent, hence the subset 
$U_\tau \subset \PML$ consisting of laminations carried by $\tau$ has nonempty interior.  Then, if 
$\Sigma \subset \euS$ is not dense, we will find a mapping class $\psi$
adapted to $\Sigma$ such that $\Sigma\in \psi U_\tau=U_{\psi\tau}$, 
and the deformation family promised in 
the theorem will then be $\psi\Omega$.

The main ingredient needed to prove Theorem \ref{T:nonrigid} is thus the following.
\begin{proposition} \label{P:magnetic family}
If $\xi(S) \geq 2$, then there exists a complete recurrent train track $\tau$ and a 
positive-dimensional family of flat structures $\Omega \subset \Flat(S)$ such that:
\begin{itemize}
\item $\tau$ is magnetic in $q$ for all $q \in \Omega$; and
\item the length of any curve $\gamma$ carried by $\tau$ is constant on $\Omega$.
\end{itemize}
\end{proposition}
\begin{proof}

If $\tau$ is a magnetic train track for $\rho$, then there is a nonnegative length vector assigned to 
each branch of $f(\tau)$.  The $\rho$--length of any curve carried by $\tau$ can be computed as the 
dot product of the weight vector for the curve with the length vector, and the allowable weight vectors
are precisely those meeting the switch conditions.
Thus, we must 
construct the family $\Omega$ so that the \textit{difference} between the length vectors for any two 
$q,q' \in \Omega$ lies in the orthogonal complement of the space of weight vectors on $\tau$.  
Geometrically, this means that the difference in length vectors for $q,q' \in \Omega$ can be 
distributed among the switches so that at each switch, the increase in length of the incoming 
branches is exactly equal to the decrease in length for each outgoing branches; see Figure 
\ref{shift}.

\begin{figure}[ht]
\setlength{\unitlength}{0.01\linewidth}
\begin{picture}(80,17)
\put(0,1){\includegraphics[width=80\unitlength]{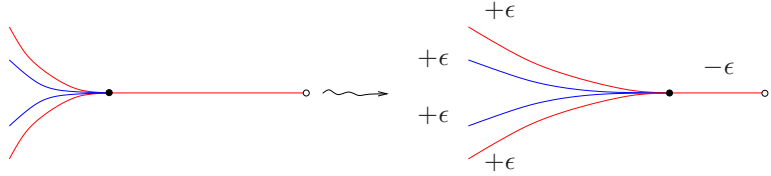}}
\put(50,0){$+\epsilon$}
\put(43,5){$+\epsilon$}
\put(43,11){$+\epsilon$}
\put(50,16){$+\epsilon$}
\put(73,10){$-\epsilon$}
%\put(11,13){$\pi$}
%\put(11,3){$\pi$}
\end{picture}
\caption{Changing the length vectors can be thought of as folding or unfolding at
switches, and leaves invariant the length of curves carried by the train
track.}
\label{shift}
\end{figure}

The idea is to build metrics and train tracks on basic building blocks, then glue them together to obtain  
$S$.  For simplicity, we only provide the details for closed surfaces in this section, as these can all simultaneously 
be handled by constructing a single building block.  To prove the theorem for all surfaces $S$ with $\xi(S) \geq 2$ 
it suffices to construct six more building blocks, using the same general ideas.  For completeness, we have 
included a description of these remaining building blocks in an appendix at the end of the paper.

\begin{figure}[ht]
\setlength{\unitlength}{0.01\linewidth}
\begin{picture}(40,39)
\put(0,0){\includegraphics[width=40\unitlength]{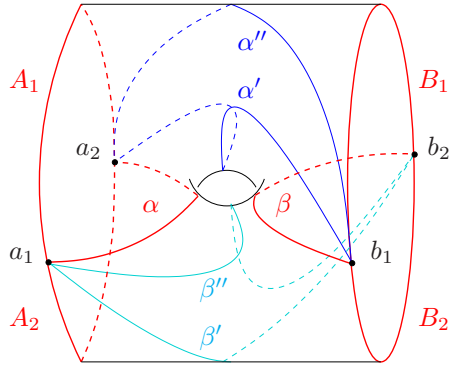}}
\put (-3,11){$a_1$}
\put (35,11){$b_1$}
\put (4,22){$a_2$}
\put (41,22){$b_2$}
\put (11,16){\color{red} $\alpha$}
\put (25,16){\color{red} $\beta$}
\put (-3,29){\color{red} $A_1$}
\put (-3,4){\color{red} $A_2$}
\put (40,29){\color{red} $B_1$}
\put (40,4){\color{red} $B_2$}
\put (17,2){\color{cyan} $\beta'$}
\put (17,7){\color{cyan} $\beta''$}
\put (21,28){\color{blue} $\alpha'$}
\put (21,33){\color{blue} $\alpha''$}
\end{picture}
\caption{One basic building block $\Delta$ and its train track $\tau$.  The cylinder $C_1$
is pictured on the top and $C_2$ on the bottom.
Copies of $\Delta$ can be glued together end to end to obtain a copy of $S$.
\label{chunk}}
\end{figure}

The basic building block $\Delta$ is a genus-one surface with two boundary components described here and shown in Figure~\ref{chunk}.  We will put a metric and a train track on $\Delta$, and then assemble $S$ from $g-1$ copies of $\Delta$ by gluing the boundary components in pairs. Choose nonperipheral arcs $\alpha$ (with endpoints $a_1,a_2$) and $\beta$ (endpoints $b_1,b_2$) joining each boundary component to itself.  Then the complement of those arcs is a pair of annuli. For any choice of $t>0$, there is a unique flat metric on $\Delta$ so that $\ell(\alpha)=\ell(\beta)=t$, and the two complementary annuli $C_i$ are Euclidean cylinders with boundary lengths $2t$ and heights $t$ (shown in Figure~\ref{combchunk}). This means each cylinder will have area $2t^2$, so $\Delta$ will have area $4t^2$.

\begin{figure}[ht]
\setlength{\unitlength}{0.01\linewidth}
\begin{picture}(80,26)
\put(0,4){\includegraphics[width=80\unitlength]{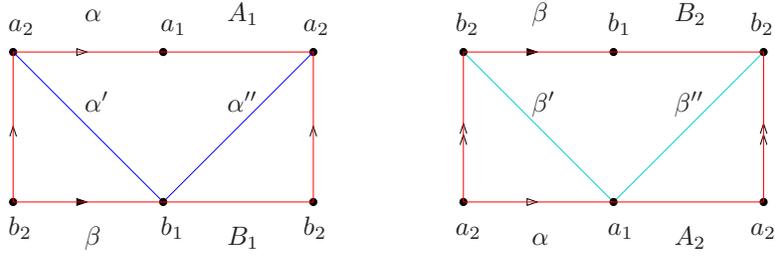}}
\put(0,1){$b_2$}
\put(16,1){$b_1$}
\put(31,1){$b_2$}
\put(8,0){$\beta$}
\put(8,14){$\alpha'$}
\put(23,14){$\alpha''$}
\put(23,0){$B_1$}
\put(47,1){$a_2$}
\put(63,1){$a_1$}
\put(78,1){$a_2$}
\put(55,0){$\alpha$}
\put(55,14){$\beta'$}
\put(70,14){$\beta''$}
\put(70,0){$A_2$}
\put(0,22.5){$a_2$}
\put(16,22.5){$a_1$}
\put(31,22.5){$a_2$}
\put(8,23.5){$\alpha$}
\put(23,23.5){$A_1$}
\put(47,22.5){$b_2$}
\put(63,22.5){$b_1$}
\put(78,22.5){$b_2$}
\put(55,23.5){$\beta$}
\put(70,23.5){$B_2$}
\end{picture}
\caption{Metric pictures of the two cylinders $C_1$ (left) and $C_2$ (right) which make up $\Delta$.}
\label{combchunk}
\end{figure}

Choose the value of $t$ so that $4t^2(g-1)=1$ (in order that the glued surface will have 
total area one). After 
gluing $g-1$ copies of $\Delta$ together end to end, we obtain a flat metric $q_0$ on $S$, 
whose singular points come from the $a_i$ and $b_i$ in the pieces $\Delta$.  
We will choose to initially glue with a quarter-twist (compare Figure 
\ref{gluepiece2}), so that there are four evenly spaced vertices around the gluing curves, and the 
singularities all have cone angle $3\pi$.

%\begin{figure}[ht]
%\setlength{\unitlength}{0.01\linewidth}
%\begin{picture}(40,40)
%\put(0,0){\includegraphics[width=40\unitlength]{gluingpiece.eps}}
%\end{picture}
%\caption{Two ends of building blocks glued with a quarter-twist.
%\label{gluepiece}}
%\end{figure}

Next we build a one-complex $T_0$ of geodesic segments in $q_0$.
In each piece $\Delta$, let $\alpha',\alpha''$ be the minimal-length
segments connecting $a_2$ to $b_1$ in $C_1$, and likewise $\beta',\beta''$ connecting $a_1$ to
$b_2$ in $C_2$ (the length of each of these will be $\sqrt{2}t$); see Figure \ref{chunk}.
Then the branches of $T_0$ are the saddle connections which belong to the boundary of a piece 
$\Delta$,
together with the arcs $\alpha,\alpha',\alpha'',\beta,\beta',\beta''$ in those pieces.
There are switches for $T_0$ at all of the singularities in the flat metric $q_0$.

\begin{figure}[ht]
\setlength{\unitlength}{0.01\linewidth}
\begin{picture}(80,29)
\put(0,0){\includegraphics[width=80\unitlength]{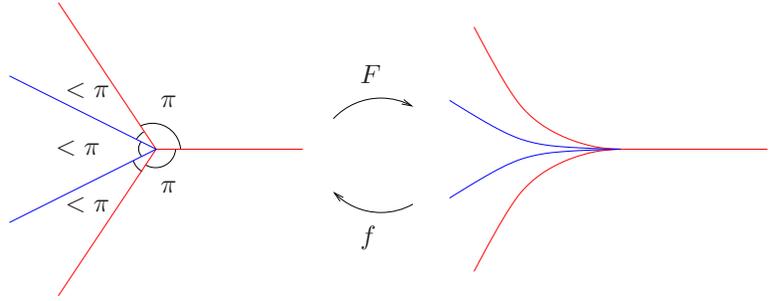}}
\put(16,11){$\pi$}
\put(16,20){$\pi$}
\put(6,21){$<\pi$}
\put(6,9){$<\pi$}
\put(5,15){$<\pi$}
\put(37,22.5){$F$}
\put(37,5.5){$f$}
\end{picture}
\caption{The homeomorphisms of $S$ pictured here map between a geodesic one-complex $T$ 
and a train track $\tau$.  
This figure shows how to use the angles in $T$ to read off the illegal turns at each switch, which specifies
the tangent spaces for $\tau$.  The inverse map $f$ is the magnetizing homeomorphism 
for $\tau$ with respect to the flat metric.\label{smoothing}}
\end{figure}

Each $1$--cell of this complex $T_0$ is smoothly embedded in $S$.  
However, there is no well-defined tangent space at the switches.  
To obtain a train track $\tau$, we apply an appropriate homeomorphism $F$ which is isotopic to the identity.  
That is, we must specify at each switch which branches are incoming and which are outgoing.
For all of the complexes $T$ in the deformation family, every switch in $T$ will have total angle $3\pi$ and
five incident branches, one of which is separated from its neighboring branches by angle $\pi$
on each side.  This determines the tangencies as in Figure \ref{smoothing}.

Any curve  $\gamma\subset\tau$ is mapped by $f=F^{-1}$ to a concatenation of geodesic segments which are 
branches of $T_0=f(\tau)$.  
But then they meet the angle conditions that suffice for geodesity (Remark~\ref{necsuffgeo}), so
$\tau$ is magnetic with respect to $q_0$.  The complementary regions are triangles, so $\tau$ is
complete, and it is straightforward to construct a positive measure on 
$\tau$, thus showing that it is recurrent.

Next we describe a deformation space $\Omega$ of $q_0$ so that 
a choice of parameters specifies a modified 1--complex $T$ (combinatorially equivalent to $T_0$
but with new lengths prescribed by the parameters) and a modified flat metric $q$, so that 
the lengths of curves carried by $\tau$ do not change as the parameters vary.
This will establish that $\tau$ remains magnetic in $q$ over the whole family $\Omega$.
The deformations can be carried out independently in each block, provided we keep track of the gluing 
information.
In each  $\Delta$, the deformations will be parameterized by two numbers 
$\epsilon$ and $\delta$ (small compared to $t$) as follows.

\begin{figure}[ht]
\setlength{\unitlength}{0.01\linewidth}
\begin{picture}(100,29)
\put(0,4){\includegraphics[width=100\unitlength]{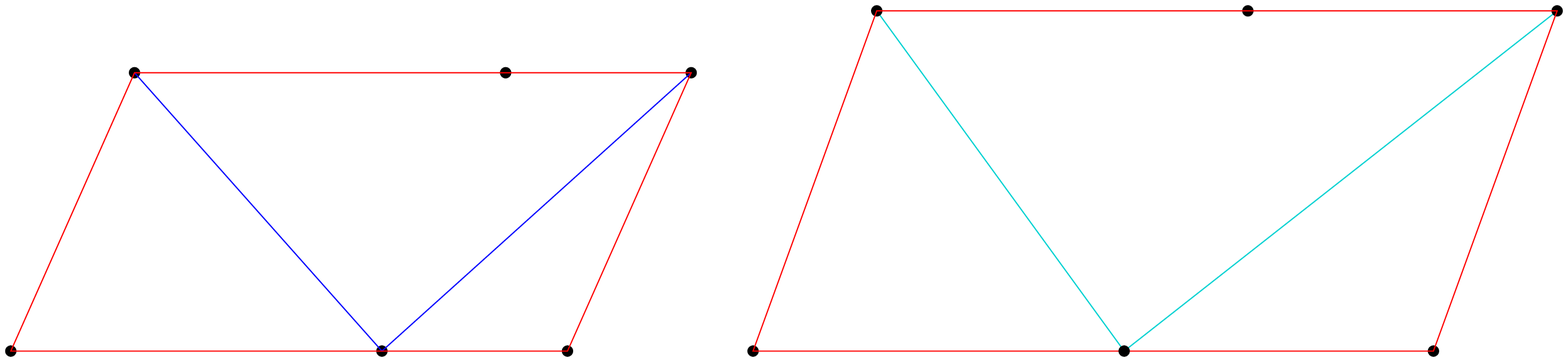}}
\put(5,0){$+\epsilon+\delta$}
\put(25,0){$+\epsilon-\delta$}
%\put(31,0){$+0$}
\put(55,0){$+\epsilon+\delta$}
\put(78,0){$-\epsilon+\delta$}
%\put(86,0){$+0$}
\put(14,25){$+\epsilon+\delta$}
\put(34,25){$+\epsilon-\delta$}
%\put(40,25){$+0$}
\put(64,28){$+\epsilon+\delta$}
\put(86,28){$-\epsilon+\delta$}
%\put(94,28){$+0$}
\put(18,13){$+2\epsilon$}
\put(26,13){$+2\epsilon$}
\put(66,15){$+2\delta$}
\put(78,15){$+2\delta$}

\end{picture}
\caption{We have two parameters $\epsilon,\delta$ to perturb
the flat structures in each piece $\Delta$. 
Metrically, this can be achieved by deforming the rectangles to parallelograms, adjusting the height and shear appropriately.  (Compare Figure \ref{combchunk}.) \label{moves}}
\end{figure}

Each new metric is built from Euclidean cylinders glued along arcs in the boundary with the same combinatorial 
pattern as that of $q_0$.  
There are four switches and $10$ arcs in $\Delta$.
To guarantee that the resulting one-complex is realizable as a train track, it is necessary that the switch
conditions be preserved; this is equivalent to requiring that the perturbations be in the row space of
the following matrix.  
$$\begin{array}{c|rrrrrrrrrr}
&A_1 & A_2 & B_1 & B_2 & \alpha & \alpha' & \alpha'' & \beta & \beta' & \beta''\\
\hline
a_1&-1 & +1 &0&0&+1&0&0&0&+1&+1\\
a_2&+1&-1&0&0&+1&+1&+1&0&0&0\\
b_1&0&0&+1&-1&0&+1&+1&+1&0&0\\
b_2&0&0&-1&+1&0&0&0&+1&+1&+1
\end{array}$$
A priori, this gives four degrees of freedom.  However, in order for the metric cylinder picture to be 
preserved, we further require two geometric conditions on the lengths of the curves: 
$$
A_1+\alpha=B_1+\beta \qquad\text{and}\qquad A_2+\alpha =B_2+\beta, 
$$
which say that the top and bottom circumferences are equal for each of $C_1$ and $C_2$.
In fact this is necessary and sufficient for the realization by metric Euclidean cylinders, 
as depicted in Figure~\ref{moves}.
(Note that the boundary components of $\Delta$ automatically have equal length because
$A_1+A_2=B_1+B_2$ holds for any perturbation satisfying the switch conditions.)

It follows that there are two free parameters, which we can 
record according to the table below.
\[ \begin{array}{c|c|c|c|c|c|c|c|c|c}
A_1 & A_2 & B_1 & B_2 & \alpha & \alpha' & \alpha'' & \beta & \beta' & \beta''\\
\hline
- \epsilon + \delta & + \epsilon - \delta & - \epsilon + \delta & + \epsilon-\delta & + \epsilon+\delta & + 2 \epsilon & + 
2 \epsilon & + \epsilon + \delta & + 2\delta & + 2 \delta \\\end{array}\]
And indeed the gluing of neighboring pieces $\Delta_i$ is also prescribed by the same parameters,
as illustrated in Figure \ref{gluepiece2}.
It is immediate, by construction, that the lengths of curves $\gamma\subset\tau$ are preserved as
these parameters vary, since changes to the length are compensated at every switch.

\begin{figure}[ht]
\setlength{\unitlength}{0.01\linewidth}
\begin{picture}(80,80)
\put(0,0){\includegraphics[width=80\unitlength]{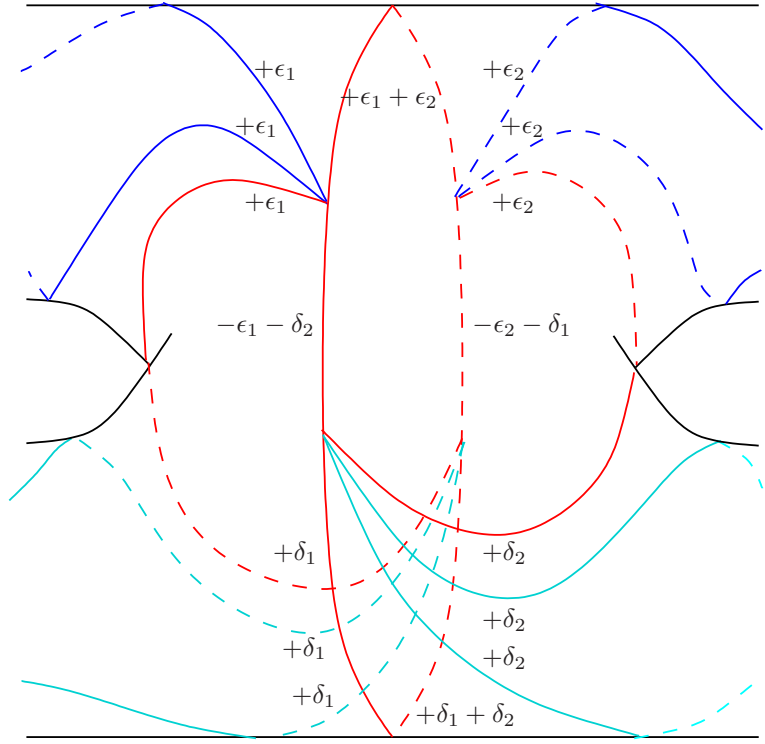}}
\put(26,70){$+\epsilon_1$}
\put(24,64){$+\epsilon_1$}
\put(25,56){$+\epsilon_1$}
\put(35,67){$+\epsilon_1+\epsilon_2$}
\put(50,70){$+\epsilon_2$}
\put(52,64){$+\epsilon_2$}
\put(51,56){$+\epsilon_2$}
\put(22,43){$-\epsilon_1-\delta_2$}
\put(49,43){$-\epsilon_2-\delta_1$}
\put(50,19){$+\delta_2$}
\put(50,12){$+\delta_2$}
\put(50,8){$+\delta_2$}
\put(43,2){$+\delta_1+\delta_2$}
\put(28,19){$+\delta_1$}
\put(29,9){$+\delta_1$}
\put(30,4){$+\delta_1$}
\end{picture}
\caption{The changes in lengths assigned to $\tau$ near one of the gluing curves.  On the left side the deformations are parameterized by $(\epsilon_1,\delta_1)$ in the block $\Delta_1$ 
and on the right by $(\epsilon_2,\delta_2)$ in the block $\Delta_2$.
\label{gluepiece2}}
\end{figure}

If we write $(\bar\epsilon,\bar\delta)=(\epsilon_1,\delta_1,\ldots,\epsilon_{g-1},\delta_{g-1})$ for 
the vector of the parameters, then we obtain a $2(g-1)$--dimensional deformation space
from the perturbed metrics $\{q_0(\bar\epsilon,\bar\delta)\}$. We let 
$$
\Omega=\{q_0(\bar\epsilon,\bar\delta)\}\cap \Flat(S),
$$ 
which is the subspace with unit area;
this has codimension $1$, so $\dim(\Omega)=2g-3$.
\end{proof}

We can now prove Theorem \ref{T:nonrigid} by finding a mapping class to apply to $\Sigma$ 
so that all of the image curves are carried by $\tau$.

\begin{proof}[Proof of Theorem \ref{T:nonrigid}.]
Let $\Omega \subset \Flat(S)$ and $\tau \subset S$ be as in Proposition \ref{P:magnetic family}.
Since $\tau$ is complete and recurrent, the subset $U_\tau \subset \PML$ consisting of those measured laminations carried by $\tau$ has nonempty interior.
Let $h\in\Mod(S)$ be a pseudo-Anosov mapping class whose attracting point in $\PML$ is a lamination $\lambda^+ \in U_\tau$.
By assumption, $\Sigma$ is not dense, so there is an open set $W\in\PML$ such that $\Sigma\cap W=\emptyset$.

Since any orbit of the mapping class group is dense in $\PML$, there is some mapping
class $\varphi\in\Mod(S)$
such that $\lambda^-\in \varphi W$, where $\lambda^-$ is the repelling
lamination of $h$.
But then $\varphi\Sigma$ misses a neighborhood
of $\lambda^-$, so for $n$ sufficiently large, any curve in $h^n \varphi\Sigma$ is carried by $\tau$.
Equivalently, any curve in $\Sigma$ is carried by $\varphi^{-1}h^{-n}\tau$.

Now we set
\[ \Omega_\Sigma=\{\varphi^{-1}h^{-n}q \, | \, q \in \Omega \},\]
and observe that the length of any curve $\gamma \in \Sigma$ is constant on $\Omega_\Sigma$ since it is carried by $\varphi^{-1}h^{-n}(\tau)$, and the property of being magnetic is clearly
preserved when both the train track and the metric are modified by the same mapping class.
\end{proof}

\begin{remark}
Here, we obtain a deformation family of dimension $2g-3$.  We make no claim 
that this is optimal, but note that the optimal dimension is bounded above and below by linear functions in 
$g$, since $\Flat(S)$ itself has dimension $12g-14$.   For the cases covered in the appendix, which 
allow punctures and boundary components, this proportionality holds 
as well:  the number of parameters in the deformation
space is linearly comparable to $g+n+b$, as is the complexity of $S$ and
therefore the dimension of $\Flat(S)$.
\end{remark}

%%%%%%%%%%%%%%%%%%%%%%%%%%%%%%%%%%%%%%%%%%%%%
\section{Flat structures as currents}
%%%%%%%%%%%%%%%%%%%%%%%%%%%%%%%%%%%%%%%%%%%%%

Bonahon's space of geodesic currents derives its utility from the fact that so many spaces embed into it in
natural ways with respect to the intersection form. For example, the space of measured laminations $\ML$, being the
completion of $\euS$ with respect to $\I$, is easily seen to embed into $\CC(S)$, and the restriction of $\I$ to
$\ML \times \ML$ is Thurston's continuous extension of geometric intersection number from weighted simple curves
to measured laminations.  In this section, we see that $\Flat(S)$ embeds naturally as well.

For closed surfaces, Bonahon constructs an embedding of $\T(S)$ into $\CC(S)$ in \cite{bon-currents} by sending a
hyperbolic metric $m$ to its associated Liouville current $L_m$. This was extended to all negatively curved Riemannian
metrics by Otal in \cite{otal} and to 
negatively curved cone metrics by Hersonsky--Paulin in \cite{her-paul}. Given any
such metric $m$, we will denote the associated current by $L_m$. 
The naturality with respect to $\I$ is expressed by the equation
\[ \I(L_m, \alpha) = \ell_m(\alpha). \]

This extends easily to $\Flat(S)$, and in fact it is possible to carry out this construction for surfaces which are not
necessarily closed.  Given $q \in \QQ(S)$, we can view $\theta \mapsto \foltheta$ as a map $\RP^1 \to \CC(S)$.

\begin{proposition} \label{P:average}
For any $q \in \Flat(S)$ there exists a current $L_q$ such that
\begin{enumerate}
\item for all $\alpha\in\euC'$, \quad $\I(L_q,\alpha)=\ell_q(\alpha)$; \\
\item for all $\mu\in\CC(S)$ and any $q\in\QQ(S)$ inducing the given
$q\in\Flat(S)$,
$$ \I(L_q, \mu) = \frac{1}{2}\int_0^\pi
\I(\foltheta,\mu) \, d \theta;
$$
\item $\I(L_q,L_q)=\pi/2$.
\end{enumerate}
\end{proposition}

\begin{proof}
We can define $L_q$ by a Riemann integral
\[ L_q = \frac{1}{2} \int_0^\pi \nu^\theta_q d \theta\]
by which we mean a limit of Riemann sums.   Since $\RP^1$ is compact, the map $f(\theta) = \foltheta$
is uniformly continuous.  As $d$ is complete, this integral exists.

For any $\alpha \in \euC'$, we recall the formula from Lemma \ref{L:length integral}
\[ \ell_q(\alpha) = \frac{1}{2} \int_0^\pi \I(\foltheta,\alpha) \ d \theta.\]  
Combining this with the uniform continuity of $\foltheta$ implies
part (1) and also part (2) for any current $\mu$ which is a scalar multiple of a current associated to a curve.  For
general currents we appeal to the density of $\mathbb R_+ \times \euC$ in $\CC(S)$ and the continuity of intersection number. Since the foliations $\foltheta$ have
$q$-length $1$ (and so $\I(L_q,\foltheta) = 1$), the third statement follows from the second statement by the computation
\begin{equation*}
\I(L_q,L_q) = \frac 12 \int_0^\pi \I(L_q, \foltheta) \, d\theta
= \frac 12 \int_0^\pi d\theta = \frac{\pi}2. \qedhere
\end{equation*}

\end{proof}

\begin{remark} In the closed case,
an equivalent definition of $L_q$ can be given as a cross-ratio, as in Hersonsky-Paulin.
\end{remark}

The embedding of $\Flat(S)$ in $\CC(S)$ is now immediate.

%%%%%%%%REPEAT
\begin{current-embedding}
There is an embedding
\[ \Flat(S) \to \CC(S)\]
denoted by $q\mapsto L_q$ so that for $q\in\Flat(S)$ and $\alpha\in\euC'$,
we have
$\I(L_q,\alpha)=\ell_q(\alpha)$.
Furthermore, after projectivizing, $\Flat(S)\to \PCC(S)$ is still an
embedding.
\end{current-embedding}

\begin{proof}
If $q_n \to q$ in $\Flat(S)$, then $\ell_{q_n}(\alpha) \to
\ell_q(\alpha)$, and hence $L_{q_n} \to L_q$ by
Theorems \ref{T:metric} and \ref{T:metric2}. Thus, $q \mapsto L_q$ is continuous.

Injectivity for $\Flat(S)\to\CC(S)$ follows directly from
Theorem~\ref{T:S rigid}, where we have shown that
 even intersection with elements of $\euS$ distinguishes flat metrics.
Injectivity for $\Flat(S)\to\PCC(S)$ follows from the fact that $\I(L_q,L_q)$
is constant, which ensures that no two currents in the image of $\Flat(S)$
can be multiples of one another.

Finally, to see that these maps are embeddings, we need only show that if $q_n$ exits every compact set in 
$\Flat(S)$, then $L_{q_n}$ has no subsequence which converges to a point of (the image of) $\Flat(S)$.  To see 
this, observe that if the lengths of all simple closed curves were bounded away from zero and infinity as 
$n\to \infty$, then $q_n$ would stay in a compact part of $\CC(S)$. 
So first suppose there exists $\gamma \in \euS$ for which 
$$
\I(L_{q_n},\gamma) = \ell_{q_n}(\gamma) \to \infty.
$$ 
In this case $L_{q_n} \to \infty$ in $\CC(S)$ and (as we show in the proof of Theorem 
\ref{T:boundary}) any projectively convergent subsequence $L_{q_n}$ must converge 
to a measured lamination, thus exiting $\Flat(S)$.  The second possibility is that there 
is a lamination $\lambda \in \ML$ with 
$$
\I(L_{q_n},\lambda) = \ell_{q_n}(\lambda) \to 0.
$$  
Since $\I(L_q,\lambda) > 0$ for any $q \in \Flat(S)$, it follows 
that any limit of $L_{q_n}$ does not lie in $\Flat(S)$.
\end{proof}

As a consequence of the embedding, we find that the length of a lamination in a flat metric is well-defined.
\begin{corollary} \label{C:flat continuous}
The flat-length function $\Flat(S)\times \euS(S) \to \R$ has a
continuous homogeneous extension
\[\overline{\ell}:\Flat(S) \times \MF(S) \to \mathbb R.\]
given by
\[(q,\mu) \mapsto \overline{\ell}_q(\mu) = \I(L_q,\mu).\]
\end{corollary}

We can now prove the main theorem.

%%%%%%%%REPEAT
\begin{rigid-iff-dense}
If $\xi(S) \geq 2$, then $\Sigma\subset\euS\subset\PMF$ is spectrally rigid over $\Flat(S)$ if
and only if $\Sigma$ is dense in $\PMF$.
\end{rigid-iff-dense}
\begin{proof}
We first assume $\Sigma$ is dense in $\PMF$.  Suppose 
$q,q' \in \Flat(S)$ have $\ell_q(\alpha) = \ell_{q'}(\alpha)$ for all $\alpha \in \Sigma$.
For any $\mu\in\MF$, the density hypothesis implies that there are scalars $t_i$ and curves $\alpha_i\in\Sigma$ such that $t_i\alpha_i\to\mu$. But 
$$
\ell_q(t_i\alpha_i)=\ell_{q'}(t_i\alpha_i),
$$ 
so Corollary \ref{C:flat continuous} implies $\ell_{q}$ and $\ell_{q'}$ agree on $\mu$.
In particular, the two metrics assign the same length to all simple closed curves. By Theorem \ref{T:S rigid}, it follows that $q= q'$, and thus $\Sigma$ is spectrally rigid.

Next assume that $\Sigma$ is not dense in $\PMF$.  Theorem \ref{T:nonrigid} implies the existence of a positive-dimensional family $\Omega_\Sigma \subset \Flat(S)$ for which the lengths of curves in $\Sigma$ is constant.  In particular, there exists a pair of distinct flat structures $q,q' \in \Omega_\Sigma$ for which $\ell_q(\alpha) = \ell_{q'}(\alpha)$ for all $\alpha \in \Sigma$, and hence $\Sigma$ is not spectrally rigid.
\end{proof}

%%%%%%%%%%%%%%%%%%%%%%%%%%%%%%%%%%%%%%%%%%%%%%
\section{The boundary of $\Flat(S)$}   \label{boundaryflat}
%%%%%%%%%%%%%%%%%%%%%%%%%%%%%%%%%%%%%%%%%%%%%%

In this section we give a description of the geodesic currents that
appear in the
closure of $\Flat(S) \subset \PCC(S)$. We will show that the limit
points have
geometric interpretations as a hybrid of a flat structure on some
subsurface and
a geodesic lamination on a disjoint subsurface
(Theorem~\ref{T:boundary}).  We call
such currents {\em mixed structures}. As a first step, we show
that the description of $L_q$ as average intersection number with
foliations
$\foltheta$ (Proposition~\ref{P:average}, part (2)) extends to any
limiting geodesic current. This description greatly simplifies the
analysis of what
geodesic currents can appear as degenerations of flat metrics.

To every nonzero quadratic differential, we consider again the map 
$$\RP^1 \to \ML(q)\subset \ML(S)$$ 
sending $\theta \mapsto \foltheta$, the foliation in direction $\theta$. We show
that given a sequence of quadratic differentials whose associated currents converge in 
$\CC(S)$, the maps $\theta \to \foltheta$ converge uniformly (up to subsequence)
to a continuous map from $\RP^1$ to $\ML(S)$.

\begin{lemma} \label{L:Lip}
For all $q\in\QQ(S)$,
$\alpha \in \euC'$, and angles $\theta_0$ and $\theta_1$, we have
$$
\Big| \I(\nu_q^{\theta_1},\alpha) - \I( \nu_q^{\theta_0},\alpha) \Big| \leq \ell_q(\alpha)\cdot |\theta_1 - \theta_0|.
$$
It follows that $\theta \mapsto \foltheta$ is Lipschitz.
\end{lemma}

\begin{proof}
Let $\omega$ be a saddle connection contained
 in a $q$--geodesic representative of $\alpha$. Assume $\omega$ has an angle $\phi$.
We have $\I(\foltheta,\omega) = \ell_q(\omega)\cdot |\sin(\theta - \phi)|$. Hence
$$
\left| \frac {d}{d \theta} \I(\foltheta,\omega) \right| = \ell_q(\omega)\cdot |\cos(\theta - \phi)| \leq
\ell_q(\omega).
$$
Integrating the above inequality from $\theta_0$ to $\theta_1$ and adding up over all saddle connections of $\alpha$ proves the lemma.
\end{proof}

\begin{proposition} \label{P:uniform convergence}
Let $q_n$ be a sequence of quadratic differentials so that $s_nL_{q_n}$
converges in $\CC(S)$ to a geodesic current $L_\infty$. Then, after possibly passing to a subsequence, the
sequence of functions
$$
f_n \from \RP^1 \to \ML(S), \qquad
f_n(\theta) = s_n\nu_{q_n}^\theta
$$
converges uniformly to a
continuous function
$$f_\infty \from \RP^1 \to \ML(S).$$
\end{proposition}

\begin{proof}
We can consider $f_n$ as maps from $\RP^1$ to $\CC(S)$. Since
$\ML(S)$ is a closed subset of $\CC(S)$, the image of the limiting map
$f_\infty$ is automatically in $\ML(S)$, provided it exists.

Equip $\CC(S)$ with the metric in Theorem \ref{T:metric} (or \ref{T:metric2} for punctured surfaces).  By the
Arzel\'a-Ascoli theorem, it is sufficient to show that the family of maps $f_n$ is equicontinuous with respect to this
metric and the union of the images have compact closure. For angles $\theta_0$ and $\theta_1$ we have
\begin{align*}
d \big(f_n(\theta_1), f_n(\theta_0) \big)
&= \sum_{\alpha \in \euC'(S)}
s_n t_\alpha \Big|\I (\nu_{q_n}^{\theta_1}, \alpha)-
\I(\nu_{q_n}^{\theta_0}, \alpha)
\Big| \\
&\le |\theta_1- \theta_0| \sum_{\alpha \in \euC'(S)} s_n t_\alpha\,
\ell_{q_n}(\alpha)\\
&= |\theta_1 - \theta_0| \cdot d(s_nL_{q_n},0).
\end{align*}
The inequality follows from Lemma \ref{L:Lip}, and the equalities are immediate from the definition the metric, together with Proposition \ref{P:average}. Since 
$$
d(s_nL_{q_n},0) \to d(L_\infty,0),
$$ 
there exists $K > 0$ such that $d(s_nL_{q_n},0) \leq K$, and so the family of maps 
$\{f_n\}$ is equicontinuous.

It remains to show that the $\cup_n f_n(\RP^1)$ has compact closure.  Observe that
\[ \I(f_n(\theta),\alpha) = \I(s_n \nu^\theta_q,\alpha) \leq s_n\ell_{q_n}(\alpha).\]
Therefore,
\[ d(f_n(\theta),0) \leq \sum_{\alpha \in \euC'(S)} t_\alpha \, \I(f_n(\theta),\alpha)
\leq \sum_{\alpha \in \euC'(S)}
s_n t_\alpha \ell_{q_n}(\alpha) = d(s_nL_{q_n},0).\]
and so $\cup_n f_n(\mathbb R \mathbb
P^1)$ is contained in the closed $K$--ball about $0$.  Since $d$ is proper, this ball is compact.
\end{proof}

We now define \textit{mixed structures} on $S$.  This requires us to first make precise what we will mean by a flat structure on a subsurface.

Suppose $X \subset S$ is a $\pi_1$--injective subsurface of $S$ with negative Euler characteristic.
We view $X$ as a punctured surface (removing every boundary component), and let $\Flat(X)$ denote the space of flat structures on $X$.
By this we mean a flat structure on each component of $X$ as described in Section \ref{S:flatstructures}, where we now require the sum of the areas of the components to be one.  Observe that the boundary curves of $X$ are realized by punctures and hence have length $0$.
Equivalently, an element of $\Flat(X)$ is given by a unit-norm quadratic differential in $\calQ(X)$, nonzero on all components, and well-defined up to multiplication by a unit-norm complex number in each component.
Representing any $q \in \Flat(X)$ by a unit-norm quadratic differential, we have the map $\RP^1 \to \ML(X)$ given by $\theta \mapsto \nu_q^\theta$ as before.
Extending measured laminations on $X$ to measured laminations on $S$ in the usual way, we can view $\theta \mapsto \nu_q^\theta$ as a map into $\ML(S) \subset \CC(S)$.

Given a subsurface $X \subset S$ as above, $q \in \Flat(S)$, and a measured lamination $\lambda \in \ML(S)$ whose support can be homotoped disjoint from $X$, we define a mixed structure $\eta = (X,q,\lambda)$ to be the geodesic current given by
\[ \eta = \lambda + \frac 12 \int_0^\pi \nu_q^\theta d\theta.\]
Here the integral is a Riemann integral, as in the proof of Proposition \ref{P:average}.  
For brevity we can write $\eta = \lambda + L_q$.
It follows that for every $\alpha \in \euC'(S)$,
\[
\I(\eta, \alpha) = \I(\lambda, \alpha) + \frac 12 \int_0^\pi
\I(\nu_q^\theta, \alpha) \, d\theta.
\]
We also allow the two degenerate situations $X = S$ and $X = \emptyset$.  In these cases, the corresponding mixed structure is a flat structure on $S$ or a measured lamination on $S$, respectively.

Now let $\Mix(S) \subset \CC(S)$ denote the space of all mixed structures, and 
$\PMix(S)$ its image in $\PCC(S)$ under the projection $\CC(S) \to \PCC(S)$.
Observe that if 
$$
\eta \in \Mix(S)\setminus\ML(S)
$$ 
then $\I(\eta,\eta) = \pi/2$, just as in Proposition \ref{P:average}.

If $\alpha$ is a curve in $\partial X$, then $\I(\nu_q^\theta,
\alpha)=0$ and
$\I(\lambda, \alpha)=0$. Hence $\I(\eta, \alpha) = 0$, although $\alpha$ may
be contained in the support of $\lambda$ (and thus $\eta$).

%%%%%%%%REPEAT
\begin{boundary}
The closure of $\Flat(S)$ in $\PCC(S)$ is exactly the space $\PMix(S)$.
That is,
for any sequence $\{q_n\}$ in $\Flat(S)$, after passing to a subsequence
if necessary,
there exists a mixed structure $\eta$ and a sequence of positive real
numbers
$\{t_n\}$ so that
$$
\lim_{n \to \infty} t_n \ell_{q_n}(\alpha) = \I(\alpha,\eta).
$$
for every $\alpha \in \euC$.  Moreover, every mixed structure is a limit
of a sequence
in $\Flat(S)$.
\end{boundary}

\begin{proof}
Let $q_n$ be a sequence of quadratic differentials such that
$t_n L_{q_n} \to L_\infty$, for positive real numbers $t_n$.
We have to show that, up to scaling, $L_\infty \in \Mix(S)$.

If the sequence $t_n$ converges to zero, then
$$
\I(L_\infty, L_\infty) = \lim_{n \to \infty} t_n^2 \I(L_{q_n}, L_{q_n})
=
\frac {\pi}{2} \lim_{n \to \infty} t_n^2= 0.
$$
That is, $L_\infty$ is a measured lamination (c.f.~Bonahon \cite{bon-currents}). Thus the theorem holds with
$X=\emptyset$.

Since every geodesic current has finite self-intersection number, we can conclude that $t_n$ does not tend to infinity.
Therefore, after taking a subsequence, we can assume that the
sequence $t_n$ is convergent, and in fact converges to $1$. That is, there is a geodesic current
(which we again denote by $L_\infty$) such that $L_{q_n} \to L_\infty$
in $\CC(S)$. Applying Proposition \ref{P:uniform convergence}
and taking a further subsequence if necessary, we can also assume that
$f_n$ converges uniformly to a continuous map $f_\infty$.
As a consequence, for every curve $\alpha \in \euC$,
$$
\I(L_\infty, \alpha) = \frac 12 \int_0^\pi \I \big(f_\infty(\theta),
\alpha \big) \, d\theta.
$$

Define $\euS_0\subset \euS$ to be the set of simple closed curves
$\alpha$ for which $\ell_{q_n}(\alpha) \to 0$.
Equivalently, $\alpha\in\euS_0$ if and only if $\I(L_\infty,\alpha)=0$.
Let $Z_0$ be the subsurface of $S$ that is filled by $\euS_0$.
That is, up to isotopy, $Z_0$ is the largest $\pi_1$--injective subsurface 
$Z$ (with respect to containment)
having the property that every closed curve in $S$ which cuts $Z$ has positive intersection number with some 
curve in $\euS_0$.
If $Z_0=S$, then there is a finite set $\alpha_1,\ldots,\alpha_k$ of curves in $\euS_0$ such that
 $\sum \alpha_i$ is a binding current, and as $L_\infty$ lies in the span of $\ML(S) \subset \CC(S)$, we have
$$ \sum \I(L_\infty,\alpha_i) >0,$$
which is a contradiction. Therefore, $Z_0$ is a proper subsurface of $S$.

We observe that, for each $\alpha_0 \in \euS_0$,
$$
 \frac 12 \int_0^\pi \I (\alpha_0,f_\infty(\theta)) \, d\theta
 = \I(L_\infty, \alpha_0) = 0.
$$
Since $f_\infty$ is continuous, this implies that $\I (\alpha_0,
f_\infty(\theta))=0$
for every $\theta$. That is, for every $\theta \in \RP^1$, the support
of
$f_\infty(\theta)$ can be homotoped to be disjoint from $Z_0$. Hence, $\I(\alpha,
L_\infty)=0$ for every
essential curve in $Z_0$. However, the restriction of $L_\infty$ to
$Z_0$ may
not be zero; for an annular component $A$ of $Z_0$, the restriction of
$f_\infty(\theta)$ to $A$ may be a measured lamination that is
supported on
the core curve of $A$.

Now choose a component $W$ of $S \setminus Z_0$. Define
$$
D(W) = \left\{
\I\left( L_\infty,\frac{\alpha}{\ell_{q_0}(\alpha)} \right)
\, \Big| \, \alpha \in \euS(W) \right\}.
$$
Observe that $D(W)$ is bounded, since $\{ \frac{\alpha}{\ell_{q_0}(\alpha)} \}$ is precompact, being contained in the compact set
\[ \{ \lambda \in \ML(S) \, | \, \ell_{q_0}(\lambda) = 1\}. \]
We argue in two cases.

\subsection*{Case 1:} $\inf ( D(W)) >0$.

In this case, we have a uniform lower bound for the $q_n$--length of any nonperipheral simple closed curve, and hence also any nonperipheral closed curve in $W$.  Since $W$ is a component of $S \setminus Z_0$, the $q_n$--lengths of the boundary curves of $W$ go to zero.
Therefore, after choosing a basepoint in $W$ (away from the boundary) and passing to a subsequence, we can assume that $q_n|_W$ converges to a flat structure on $W$ \textit{geometrically}, that is, after re-marking by a homeomorphism. (See Appendix A of \cite{McMullen} for a thorough
discussion of the geometric topology on the space of quadratic differentials. In particular, McMullen
establishes the existence of the relevant geometric limit in his Theorem A.3.1 for points in moduli space.)
Since any given curve in $W$ has a uniform upper bound to its $q_n$--length, we may assume that the 
re-marking homeomorphisms are isotopic to the identity in $W$, and hence $q_n|_W$ converges to a flat structure on $W$ (though not necessarily of unit area).

\subsection*{Case 2:} $\inf( D(W) )=0$.

In this case, we have a sequence of simple curves $\alpha_n \in \euC(W)$
such that
$$\lim_{n\to\infty}\I\left(L_\infty,\frac{\alpha_n}{\ell_{q_0}(\alpha_n)}\right)=0.$$
Since $\{ \frac{\alpha_n}{\ell_{q_0}(\alpha_n)} \}$ is precompact, we may pass to a subsequence
so that
$$\frac{\alpha_n}{\ell_{q_0}(\alpha_n)} \to \lambda,$$
for some lamination $\lambda$. The continuity of intersection number
implies
$\I(L_\infty, \lambda) =0$.

We observe that $\lambda$ has to fill $W$.  To see this, let $W' \subset W$ be the subsurface filled by $\lambda$.  Since $\I(L_\infty,\lambda)=0$, it follows that $\I(f_\infty(\theta),\lambda) = 0$.  Therefore, $\I(f_\infty(\theta),\partial W') = 0$ and hence $\I(L_\infty,\partial W') = 0$.  Thus $\partial W' \in \euS_0$ and $W = W'$.
The support of $L_\infty$ consists of geodesics having no transverse
intersection with the support of $\lambda$. Therefore, the support of
$L_\infty$, restricted to $W$, equals the support of $\lambda$.
That is, $L_\infty|_{W}$ is a (filling) measured lamination in $W$.

We have shown that $L_\infty$ is a mixed structure $(X,q,\lambda)$ where $X$ is the union of all 
$W$ as in Case 1,
$q$ is the limiting flat structure in $X$ and $\lambda$ is the union of limiting
laminations in Case 2 and weighted curves from all the annular components
$A$ where the restriction of some $f_\infty(\theta)$ to $A$ is
nontrivial.  Since 
$$
\I(L_\infty,L_\infty) = \lim_{n \to \infty} \I(L_{q_n},L_{q_n}) = \pi/2,
$$ 
the sum of the areas of the flat structures is $1$.

To finish the proof, we show that any mixed structure $\eta=(X,q,\lambda)$ appears as the limit of 
a sequence of flat structures.  The idea is to build the metric from $q$ on $X$, by making 
small slits at the punctures and gluing in a sequence of metrics on the complement,
limiting to $\lambda$ and with area tending to zero.

First lift $q$ to an arbitrary representative $q\in\QQ(X)$.  Next
write the lamination $\lambda$ as $\lambda = \lambda_0 + \lambda_1$, where 
$\lambda_0$ is supported on a disjoint union of simple closed curves, and 
$\lambda_1$ has support a 
lamination with no closed leaves.  We can further decompose $\lambda_0 = \sum_i s_i \alpha_i$ 
for some $\alpha_i \in \euS$ and $s_i > 0$.  For each $i$ and all $n \geq 0$, let $C_{i,n}$ be a 
Euclidean cylinder with height $s_i$ and circumference $2/n^2$.
Let $Y$ be the subsurface filled by $\lambda_1$ (with boundary replaced by punctures) and let 
$q' \in \QQ(Y)$ be any quadratic differential for which $\nu_{q'}^0 = \lambda_1$.  Consider the 
Teichm\"uller deformation $A_n q'$, where
\[A_n = \left( \begin{array}{cc} n & 0\\ 0 &\frac{1}{n}\\ \end{array} \right).\]
This tends to the vertical foliation of $q'$ (which was chosen to be $\lambda_1$)
by an argument which appears in Proposition~\ref{P:affinepaths}.

Let $Z$ be the union of the nonannular, non-pants components of 
$S \setminus (X \cup Y)$.  Choose any quadratic differential $q'' \in \QQ(Z)$ for which 
the vertical foliation is minimal (for simplicity).

Now we construct a flat structure $q_n$ as follows.  At each puncture in $q$ that corresponds to an essential curve in $S$ (that is, a boundary component of $X$ in 
$S$) we cut open a slit of size $1/n^2$ emanating from the given puncture, in any 
direction. Similarly, letting 
$$
q'(n) = \frac{1}{n}A_n q' \qquad\text{and}\qquad q''(n) = \frac{1}{n} q'',
$$ 
cut open slits of length $1/n^2$ along the vertical foliations of each, one starting at 
each of the punctures of $Y$ and $Z$ that correspond to essential curves in $S$.
Note that since the vertical foliations of $q'(n)$ and $q''(n)$ are minimal, these constructions are possible.
We glue these and the cylinders $\{C_{i,n}\}$ along their boundaries to recover the surface $S$ with a quadratic differential  $q_n$, which we scale to have 
unit norm (as $n$ tends to infinity, the areas of $q'(n)$ and $q''(n)$ go to zero and
the scaling factor tends to $1$). We glue along the boundaries by a local isometry, 
and if we further require the relative twisting of $q_0$ and $q_n$ along every gluing 
curve to be uniformly bounded, we obtain a sequence limiting to 
$\eta$ in $\PCC(S)$, as desired.
\end{proof}

\subsection*{A dimension count}
The Thurston boundary is very nice as a topological space:  it is 
a sphere compactifying a ball, having  codimension one in the compactification
$\overline{\T(S)}$.  Here, we show
that the codimension of $\partial\Flat(S)$ is three.
To see this, first recall that for a connected surface
$S$ of genus $g$
with $n$ punctures, $\T(S)$ is $(6g+2n-6)$--dimensional.
The space $\calQ(S)$ of quadratic differentials on $S$ has twice the
dimension and $\Flat(S)$ is a quotient of
$\calQ(S)$ by an action of $\C$. Hence
$$ \dim(\Flat(S))= 12g+4n-14.$$

For any $\pi_1$--injective subsurface $Y \subset S$, we consider the subset $\partial_Y \subset \partial \Flat(S)$
consisting of those $\eta = (X,q,\lambda)$ for which the support of the flat metric is
$X = S \setminus Y$.  Observe that $\partial \Flat(S)$ is a disjoint union of subsets of the form $\partial_Y$, as $Y$ varies over subsurfaces of $S$.  In the case that $Y$ is an annulus with core curve $\alpha$, we simply write $\partial_Y = \partial_\alpha$.  Points in $\partial_\alpha$ are projective mixed structures of the form $w \alpha + L_q$, where $q \in \Flat(X)$ and the weights $w$ on $\alpha$ are nonnegative numbers.
We first compute the dimension of the sets $\partial_\alpha$.

If $\alpha$ is a non-separating curve, then $X$ is connected, has
genus one less than $S$ and has $2$ extra punctures. That is,
$$\dim(\Flat(X)) = 12(g-1) + 4(n+2) -14 = 12g+4n-18.$$
To recover the space $\partial_\alpha$, we restore one extra dimension
from the weight on $\alpha$, so that $\dim(\partial_\alpha)=12g+4n-17$, giving that space 
codimension three with respect to $\Flat(S)$.

Now let $\alpha$ be a separating curve. Then $X=X_1 \cup X_2$, where
$X_i$
is a surface of genus $g_i$ with $n_i$ punctures ($i=1,2$) so
that
$g=g_1+g_2$ and $n=n_1+n_2+2$. Therefore, $\calQ(X)$ has dimension
$$(12g_1+4n_1-12)+(12g_2+4n_2-12)=12g+4(n+2)-24=12g+4n-16.$$
The space $\Flat(X)$ is the quotient of $\calQ(X)$
by scaling and rotation in each component, but the total area must be one in the end,
giving
$$\dim(\Flat(X))= \dim(\calQ(X))-3=12g+4n-19.$$
The space $\partial_\alpha$ has one extra dimension from the weight on $\alpha$ and is
$(12g+4n-18)$--dimensional.  In the separating case, then, the codimension
is four with respect to $\Flat(S)$.

It is not difficult to see that for larger-complexity subsurfaces $Y \subset S$, the subsets 
$\partial_Y$ have higher codimension in $\Flat(S)$, since for any subsurface $W$, 
$$\dim\ML(W) < \dim\Flat(W).$$
Since $\partial \Flat(S)$ is a countable union of sets of the form $\partial_Y$, each of which can 
be exhausted by compact (hence closed) sets, the dimension of $\partial \Flat(S)$ is the 
maximum dimension of any subset $\partial_Y$ (by the Sum Theorem in \cite{nagata}), 
which is therefore $12g+4n-16$.  
So we have seen that $\partial \Flat(S)$ has codimension three in 
$\overline{\Flat(S)}$.

%%%%%%%%%%%%%%%%%%%%%%%%%%%%%%%
\section{Remarks and questions} \label{S:Finalremarks}
%%%%%%%%%%%%%%%%%%%%%%%%%%%%%%%

%%%%%%%%%%%%%%%%%%%%%%%

\subsection{Rigidity for closed curves}

Though we have a complete description of rigidity for
$\Sigma\subset\euS$,
the more general case of $\Sigma\subset\euC$ is still open.

We have already seen a sufficient condition for $\Sigma\subset\euC$ to be
spectrally rigid over flat metrics:  clearly if 
$\PMF\subset\calP(\overline\Sigma)$, then $\Sigma$ is spectrally
rigid because its lengths determine all those from $\euS$ in that case.
Here is a further observation.

\begin{proposition}
If $\overline\Sigma$ has nonempty interior as a subset of $\CC(S)$, then
$\Sigma$ is spectrally rigid over any class of metrics that embeds
naturally into $\CC(S)$.
\end{proposition}

\begin{proof}
Fix a pair of currents $\nu_1$,$\nu_2$ and set 
$$
f(\mu):=\I(\nu_1,\mu)-\I(\nu_2,\mu).
$$
Suppose there is an open set in $f^{-1}(0)$ containing a current $\mu_0$.
Let $\{\delta\}$ be the set of currents close to the identity in the metric
on $\CC(S)$ (defined in Theorems~\ref{T:metric},\ref{T:metric2}).
Then for $\delta$ sufficiently close to the zero current,
$f(\mu+\delta)=0$, so $f(\delta)=0$ by linearity of $\I$.
But every current is a multiple of a small current and $f$ is linear, so this shows that
$\nu_1$ and $\nu_2$ have the same intersection number with all of the elements of $\CC(S)$.
We can conclude that  $\nu_1=\nu_2$ by Otal's theorem.
In fact, we have shown that intersections with any open set of currents suffice to separate points in $\CC(S)$.

To apply this to a class of metrics such that $\calG(S)\hookrightarrow \CC(S)$
and $\I(L_\rho,\alpha)=\ell_\rho(\alpha)$,
suppose that $\lambda_\Sigma(\rho)=\lambda_\Sigma(\rho')$.
Letting 
$$
\nu_1=L_\rho \qquad\text{and}\qquad \nu_2=L_{\rho'},
$$ 
we have $f(\mu)=0$ for all $\mu\in\overline\Sigma$, which contains an open set 
by assumption.  This then implies that $\rho=\rho'$.\end{proof}

%%%%%%%%%%%%%%%%%%%%%%%

\subsection{Remarks on the boundary of $\Flat(S)$}

\begin{remark}
We observe that \Teich geodesics behave well with respect
to the compactification of $\Flat(S)$.  For all the points along a \Teich
geodesic, the vertical and horizontal foliations are constant, up to scaling.
In this compactification, every geodesic limits to its vertical foliation.

\begin{proposition} \label{P:affinepaths}
Let $\G \from \R \to \T(S)$ be a \Teich geodesic, let $q_t$ be
the corresponding quadratic differential at time $t$ and
$\nu_0$ be the initial vertical foliation at $q_0$. Then,
considering $\nu_0$ as an element of $\CC(S)$, we have
$$
\frac{L_{q_t}}{e^t} \to \nu_0.
$$
\end{proposition}

\begin{proof}
The flat length of a curve is less than the sum of its horizontal length
and its vertical length and is larger than the minimum of its
horizontal and vertical lengths. That is, if $\mu_t$ and $\nu_t$ are the horizontal and the
vertical foliation at $q_t$ then for every $\alpha \in \euC'(S)$ we have
$$
\min \big(\I(\alpha, \nu_t), \I(\alpha, \mu_t)\big) \leq \ell_{q_t}(\alpha)
\leq \I(\alpha, \nu_t) +\I(\alpha, \mu_t).
$$
But $\I(\alpha, \nu_t) = e^t \I(\alpha, \nu_0)$ and
$\I(\alpha, \mu_t) = e^{-t} \I(\alpha, \mu_0)$. Therefore,
$$
\frac{\I(L_{q_t}, \alpha)}{e^t} = \frac{\ell_{q_t}(\alpha)}{e^t} \to
\I(\alpha, \nu_0).
$$
Theorems~\ref{T:otal} and \ref{T:like otal} assure us that a current is completely determined by these intersections.
\end{proof}

This proposition shows not only that points along a \Teich geodesic converge
to a unique limit in $\partial \Flat(S)=\PMix(S)$, but also that
different geodesic rays with a common basepoint have different limit points in the boundary
(because they have different vertical foliations). This is in contrast with the
situation
for the Thurston boundary where both of the above statements are false
(see \cite{lenzhen} and \cite{masur:CG}).
\end{remark}

\begin{remark}
The boundary of $\Flat(S)$ described here and the Thurston boundary of \Teich 
space are compatible in a certain sense. Consider the projection 
$$\sigma:\Flat(S)\to \T(S)$$ 
which sends a flat metric $q$ to the hyperbolic metric in its conformal class.
As flat structures degenerate to the boundary, the corresponding hyperbolic metrics accumulate in $\PML$.
The following proposition describes the relationship between the limiting structures:
they have zero intersection number.  The results on \Teich geodesics in the previous
remark illustrate a special case of this.

\begin{proposition}
Let $q_n$ be a sequence of flat structures on $S$ and $\sigma_n=\sigma(q_n)$.
Assume that $\sigma_n \to \mu$ in the Thurston compactification and $q_n \to \eta$ in $\PCC(S)$, where $\mu$ is a geodesic lamination and $\eta$ is a mixed structure in 
$\partial\Flat(S)$. Then 
$$\I(\mu, \eta)=0.$$
\end{proposition}

\begin{proof}
We suppose that $s_n q_n \to \eta$ as currents and $t_n \ell_{\sigma_n}(\nu) \to \I(\mu,\nu)$ for all $\nu \in \ML(S)$.
Since the $\sigma_n$ and $q_n$ escape from $\T$ and $\Flat(S)$, respectively, we know that the $t_n$ tend
to zero and the $s_n$ are bounded.
There is a sequence of approximating laminations $\mu_n$ to $\sigma_n$ such that $t_n\mu_n\to \mu$ in $\ML(S)$ and $\I(\mu_n,\nu) \le \ell_{\sigma_n}(\nu)$ for all $\nu \in \ML(S)$; see \cite[Expos\'e 8]{FLP}. Then we have
\begin{align*}
\I(\mu,\eta) 
&=  \limni \I(t_n\mu_n,s_n L_{q_n}) \\ & \\
&=  \limni \frac12 \int_0^\pi \I\left(t_n\mu_n,s_n\nu_{q_n}^\theta \right) \, d\theta\\ & \\
&\le \limni \frac12 \int_0^\pi t_n \ell_{\sigma_n}\left(s_n\nu_{q_n}^\theta \right)\, d\theta.
\end{align*}

We also have that $\ell_{\sigma_n}(\nu_{q_n}^\theta)$ is bounded above by $\sqrt{A\cdot \Ext_{[\sigma_n]}(\nu_{q_n}^\theta)}$,
where $A$ is the $\sigma_n$-area of $S$, which is a constant. This is true for simple closed curves by definition of extremal length, and holds for laminations because both hyperbolic length and extremal length extend continuously to $\ML(S) = \MF(S)$; see \cite{Kerckhoff}.
Furthermore, extremal length of $\nu_{q_n}^\theta$ is realized in the quadratic differential metric
for which the foliation is straight, namely $q_n$.   Finally, since $\ell_{q_n}(\nu_{q_n}^\theta)=1$
and since the product $t_n s_n$ tends to zero, we conclude that $\I(\mu,\eta)=0$,
as desired.\end{proof}

Note that if $\rho$ is any metric in the conformal class of $q$ to which a current $L_\rho$ can be naturally associated, the extremal length argument gives us that
$\I(L_\rho,L_q) \le \frac{\pi}{2} \sqrt A$, for $A$ the $\rho$-area of the surface.  This gives an even simpler proof of the previous theorem 
for the case of closed surfaces $S$ by taking $\rho$ to be the hyperbolic metric in the conformal class of $q$.
Furthermore, we also have the following interesting inequality:
\[ \I(L_q,L_{q'})\le 1 \]
where $q$ and $q'$ are any two flat metrics in the same fiber over $\T(S)$.

\end{remark}

\begin{remark}
The boundary for $\Flat(S)$ can be used to
construct a boundary for $\QQ(S)$. We have shown that, for a
sequence $q_n$ of quadratic differentials, after taking a subsequence,
not only $L_{q_n}$ converge in $\CC(S)$, but by Proposition \ref{P:uniform convergence} 
the maps $f_n(\theta)$
converge uniformly to a map $f_\infty$, after appropriate scaling. One can equip the space
$$
\Big\{ (\mu,f) \, \Big|\,  \mu \in \CC(S), \quad f\from \RP^1 \to \ML(S)
\quad \text{continuous} \Big\}
$$
with the product topology, from $\CC(S)$ in one factor and uniform convergence in the other.
Then the map $q_n \mapsto (L_{q_n}, f_n)$ is an embedding and has
compact
closure in the projectivization. However, it seems difficult to describe which pairs
$(\mu, f)$ appear in the boundary of $\QQ(S)$.
\end{remark}

%%%%%%%%%%%%%%%%%%%%%%%
\subsection{Unmarked length spectrum does not suffice}
\label{sunada}

The Sunada construction of distinct isospectral hyperbolic
surfaces, originally put forward in ~\cite{sunada},
is easily applied to metrics in $\Flat(S)$.  We briefly sketch the idea.

Sunada constructs non-isometric hyperbolic surfaces
$S_1$,$S_2$ covering
a common $S$  by choosing ``almost-conjugate'' subgroups 
$\Gamma_1,\Gamma_2$ of $\pi_1(S)$ and lifting to corresponding covers.  
(Almost-conjugacy means that each conjugacy class of $\pi_1(S)$ intersects
the two subgroups in the same number of elements.)
If a flat metric $q$ is placed on $S$, then the
argument that its lifts $q_1,q_2$ are iso-length-spectral runs
exactly as for the hyperbolic metrics:
$\Lambda_\euC(q_1)=\Lambda_\euC(q_2)$,
because an element of $\pi_1(S)$ conjugating $\gamma_1\in\Gamma_1$
to $\gamma_2\in\Gamma_2$ associates a geodesic
of $q_1$ for which the associated deck transformation is $\gamma_1$ 
to an equal-length $q_2$--geodesic by acting on the lift to $\tilde S$.
(See~\cite{buser} for a careful discussion.)

The key in using the Sunada construction is therefore to find examples for which 
the metrics on $S_1$ and $S_2$ are not isometric, but such choices of hyperbolic 
metrics on $S$ are in fact generic.
Now put a flat metric $q$ on $S$ in the conformal class of
such a hyperbolic metric, and lift it to flat metrics $q_i$ on $S_i$.
If $q_1$ is isometric to $q_2$, then they are conformally equivalent, so
the corresponding hyperbolic metrics are equal, a contradiction.  
Thus there is a ready supply of examples of distinct flat metrics for which 
$\Lambda_\euC(q_1)=\Lambda_\euC(q_2)$.

Note that this argument is for the unmarked length spectrum $\Lambda_\euC$ of all closed
curves; the counts of lifts in the Sunada construction are not sensitive to whether curves are 
simple.  The question of whether there are distinct flat surfaces with equal
unmarked length spectrum for the simple closed curves $\euS$ remains open.

\section{Appendix: More building blocks.}

Here we sketch the construction of the remaining basic building blocks needed to carry out the proof of Proposition \ref{P:magnetic family} for a general surface $S$ with $\xi(S) \geq 2$.  The building blocks are surfaces $\Sigma_{g,n,b}$ where $g$ is the genus, $n$, 
is the number of punctures/marked points, and $b$ is the number of boundary components.  If we let $d$ denote the dimension of the space of metrics we construct on $\Sigma_{g,n,b}$, then the resulting pairs $(\Sigma_{g,n,b},d)$ are:
\[ (\Sigma_{1,0,2},2) \, , \, (\Sigma_{1,0,1},2) \, , \, (\Sigma_{1,1,1},2) \, , \, (\Sigma_{0,2,1},0) \, , \, (\Sigma_{0,3,1},2) \, , \, (\Sigma_{0,4,1},2) \, , \, (\Sigma_{0,4,2},3). \]
The case of $\Sigma_{1,0,2}$ was discussed in Section \ref{S:families}.
Each building block will come equipped with a train track that carries the boundary, and when the building blocks are assembled to construct the surface $S$, the train tracks assemble to a complete recurrent train track.  The family of metrics for each will keep the boundary length fixed, so that the deformations can be carried out independently on each piece.  Gluing together the deformations is carried out in a fashion similar to that used for the closed case in Section \ref{S:families}.

By gluing the pieces above, one can construct flat structures and magnetic 
train-tracks on any surface with $\xi \geq 2$.  In sketch, one can attach along their boundaries copies of 
$\Sigma_{1,0,2}$ (to add one to the genus) and $\Sigma_{0,4,2}$  (to add four punctures) 
to obtain a surface 
with two boundary components which has almost all the required genus and 
number of punctures. One then caps off the two boundaries with the appropriate 
pieces to obtain the desired surface.  The resulting flat structure, $q$, and
magnetic train track, $\tau$, has a deformation for which the length of every curve 
in $\tau$ remains constant. The dimension of this deformation space is at least 
equal to the sum of the number of allowable deformations for the pieces 
involved.  As before, keeping total area one imposes a codimension-one condition at the end.

\subsection{$(\Sigma_{1,0,1},2)$}

The topological picture of $\Sigma_{1,0,1}$ together with its train track are 
shown on the left in Figure \ref{chunk101}.

\begin{figure}[ht]
\setlength{\unitlength}{0.01\linewidth}
\begin{picture}(80,30)
\put(0,0){\includegraphics[width=80\unitlength]{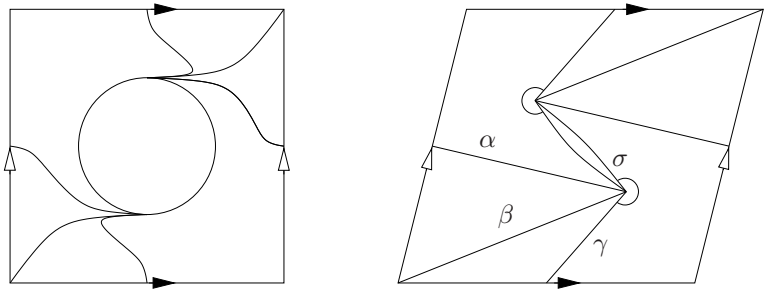}}
\put(50,15){$\alpha$}
\put(52,7){$\beta$}
\put(62,4){$\gamma$}
\put(64,13){$\sigma$}
\end{picture}
\caption{The topological picture of $\Sigma_{1,0,1}$ and its train track is on 
the left and the metric picture is on the right.  The angles of at least $\pi$ are indicated in the metric picture.}
\label{chunk101}
\end{figure}

The arcs in the boundary of the square are identified in pairs as indicated by the 
arrows. The generic metric in the deformation family is shown on the right in 
Figure \ref{chunk101} and is described as follows.  Starting with a parallelogram 
having one horizontal side and one skew side with positive slope, we identify the 
opposite sides by a translation as indicated by the arrows.  Next we cut a slit along a 
geodesic arc $\sigma$ in the parallelogram, and we assume that $\sigma$ has 
negative slope.  This produces a metric version of the topological surface, and the
geodesic version of the train track is obtained by adding the arcs $\alpha,\beta,\gamma$ 
as indicated.

If we require the boundary length to be fixed, so the length of $\sigma$ is fixed, then 
the dimension of the space of all such metrics is $4$: there are $3$ dimensions for 
the parallelogram and one for the angle $\sigma$ makes with the horizontal side.  
We now wish to impose constraints which guarantee that the change in lengths of 
the branches can be distributed to the switches in such a way that at each switch the 
increase in the lengths of the incoming branches is equal to the decrease in lengths of 
the outgoing branches.  In this case, one checks that this can only be accomplished if 
each of the lengths of $\alpha, \beta,\gamma$ change by the same amount.  This 
imposes two conditions: the difference in lengths of $\alpha$ and $\beta$ is constant, 
and the difference in lengths of $\beta$ and $\gamma$ is constant.  This cuts the 
dimension of the deformation space down by two, resulting in the 2--dimensional space 
of deformations that was claimed. It is interesting to note that in this case, there are 
nontrivial deformations for which the length vector on the train track itself remains 
constant.

\subsection{$(\Sigma_{1,1,1},2)$}

This building block is obtained by a minor modification of the previous one; see 
Figure \ref{chunk111}. We leave the details to the reader, but point out one new 
feature in this example not present in the previous two pieces.
Namely, the map $f:(\hat S,P) \to (\hat S,P)$ in the definition of a magnetic train 
track cannot be taken to be a homeomorphism. This is because the small branch 
that partially surrounds the puncture is collapsed to a point---the length vector assigns 
this branch zero length.

\begin{figure}[ht]
\setlength{\unitlength}{0.01\linewidth}
\begin{picture}(80,30)
\put(0,0){\includegraphics[width=80\unitlength]{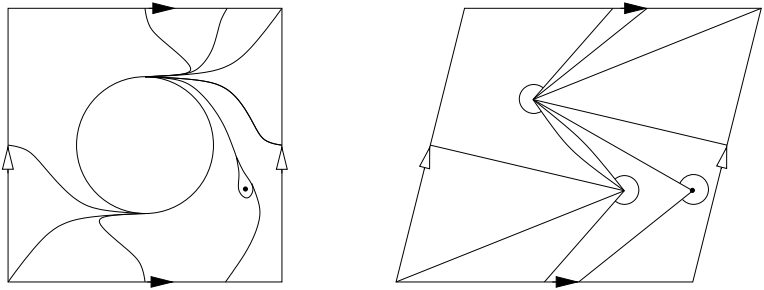}}
\end{picture}
\caption{The case $\Sigma_{1,1,1}$ is a minor variation of $\Sigma_{1,0,1}$ shown in Figure 
\ref{chunk101}.}
\label{chunk111}
\end{figure}

\subsection{$(\Sigma_{0,2,1},0)$}

For this building block, the metric picture degenerates completely to an arc and there is ``no room''  to construct any deformations; see Figure \ref{chunk021}.
This piece is used to cap off boundary components.  
The metric effect is simply to glue the boundary component to itself.

\begin{figure}[ht]
\setlength{\unitlength}{0.01\linewidth}
\begin{picture}(40,25)
\put(0,0){\includegraphics[width=40\unitlength]{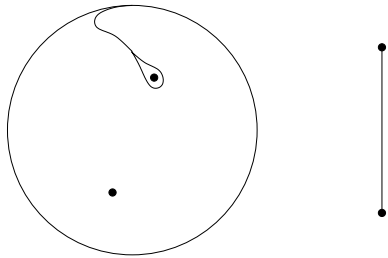}}
\end{picture}
\caption{The metric version for $\Sigma_{0,2,1}$ degenerates.}
\label{chunk021}
\end{figure}

\subsection{$(\Sigma_{0,3,1},1)$}

The generic metric is obtained from a parallelogram by identifying the arcs in the 
sides as indicated by the arrows in Figure~\ref{chunk031} via an appropriate 
semi-translation, then cutting open a slit in the interior emanating from one of the 
marked points. The small-loop branches of the train track are assigned zero length, 
and the three main branches (not in the boundary) are represented by the darkened 
arcs in the metric picture. A dimension count as above reveals that the deformation 
space has dimension $2$.

\begin{figure}[ht]
\setlength{\unitlength}{0.01\linewidth}
\begin{picture}(80,25)
\put(0,0){\includegraphics[width=80\unitlength]{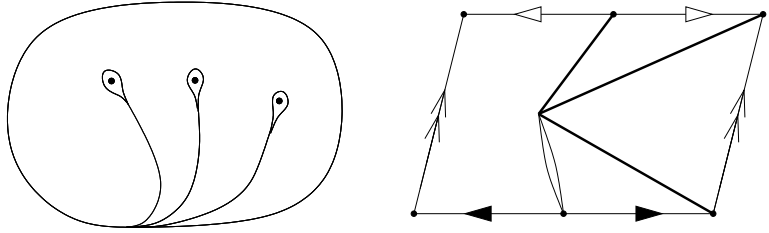}}
\end{picture}
\caption{The dark lines in the metric picture for $\Sigma_{0,3,1}$
represent the image of the three main branches of the train track.}
\label{chunk031}
\end{figure}

\subsection{$(\Sigma_{0,4,1},1)$}

This building block is obtained from the previous one in a similar fashion to the way $\Sigma_{1,1,1}$ is obtained from $\Sigma_{1,0,1}$; see Figure \ref{chunk041}.
We leave the details to the reader.

\begin{figure}[ht]
\setlength{\unitlength}{0.01\linewidth}
\begin{picture}(80,25)
\put(0,0){\includegraphics[width=80\unitlength]{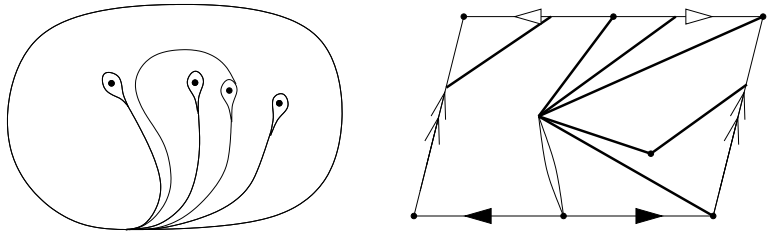}}
\end{picture}
\caption{Adding another puncture to $\Sigma_{0,3,1}$ to produce $\Sigma_{0,4,1}$.}
\label{chunk041}
\end{figure}

\subsection{$(\Sigma_{0,4,2},3)$}
The metric picture is formed from a parallelogram with sides identified as illustrated, then slit open along two equal-length arcs as shown.
We have labeled some of the branches of the train track in Figures \ref{chunk042a} and  \ref{chunk042c}.

\begin{figure}[ht]
\setlength{\unitlength}{0.01\linewidth}
\begin{picture}(60,38)
\put(0,0){\includegraphics[width=60\unitlength]{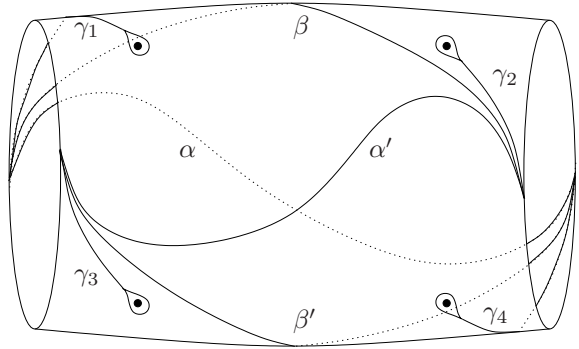}}
\put(18,20){$\alpha$}
\put(38,20){$\alpha'$}
\put(30,33){$\beta$}
\put(30,2){$\beta'$}
\put(7,33){$\gamma_1$}
\put(51,28){$\gamma_2$}
\put(7,7){$\gamma_3$}
\put(50,3){$\gamma_4$}
\end{picture}
\caption{The topological picture of $\Sigma_{0,4,2}$ together with its train track.}
\label{chunk042a}
\end{figure}

\begin{figure}[ht]
\setlength{\unitlength}{0.01\linewidth}
\begin{picture}(50,31)
\put(0,0){\includegraphics[width=50\unitlength]{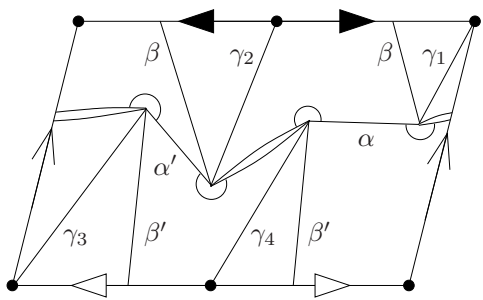}}
\put(15.5,13){$\alpha'$}
\put(37,16){$\alpha$}
\put(14.5,25){$\beta$}
\put(39,25){$\beta$}
\put(14.5,6){$\beta'$}
\put(31.7,6){$\beta'$}
\put(43.7,25){$\gamma_1$}
\put(23.5,25){$\gamma_2$}
\put(6,6){$\gamma_3$}
\put(25.7,6){$\gamma_4$}
\end{picture}
\caption{The metric picture of $\Sigma_{0,4,2}$.}
\label{chunk042c}
\end{figure}

The space of allowable deformations has dimension $3$.  To see this, first note 
that to properly distribute length changes at the switches, the lengths of the arcs 
are allowed to vary according to the following:
\[ \begin{array}{c|c|c|c|c|c|c|c}
\alpha & \alpha' & \beta & \beta' & \gamma_1 & \gamma_2 & \gamma_3 & \gamma_4\\
\hline
+ \epsilon + \delta & + \epsilon + \delta & + \epsilon + \delta & + \epsilon+\delta 
& + \epsilon & + \delta & + \epsilon & + \delta \\
\end{array} \]

To see that these variations are indeed possible (for small $\epsilon$ and $\delta$), 
we again appeal to a dimension count. The space of parallelograms with a pair of slits 
of fixed, equal length is $3+3+3 = 9$ dimensions. The 9-dimensional parameter space 
is subject to 6 equations derived from the geometry, 
leaving $3$ degrees of freedom.

\bibliography{newgen}{}
\bibliographystyle{plain}
\end{document}